\documentclass{amsart}

\usepackage{amsmath,amssymb}

\newtheorem{lemma}{Lemma}[section]
\newtheorem{prop}[lemma]{Proposition}
\newtheorem{cor}[lemma]{Corollary}
\newtheorem{thm}[lemma]{Theorem}
\newtheorem{example}[lemma]{Example}
\newtheorem{thm?}[lemma]{Theorem?}

\newtheorem{ques}[lemma]{Question}

\newtheorem{remark}[lemma]{Remark}

\newcommand{\F}{\mathbb{F}}

\newcommand{\ra}{\ensuremath{\rightarrow}}

\newcommand{\Z}{\mathbb{Z}}
\newcommand{\N}{\mathbb{N}}
\newcommand{\R}{\mathbb{R}}

\newcommand{\C}{\mathbb{C}}
\newcommand{\Q}{\mathbb{Q}}

\newcommand{\Aut}{\operatorname{Aut}}

\newcommand{\GL}{\operatorname{GL}}

\newcommand{\tors}{\operatorname{tors}}

\newcommand{\disc}{\operatorname{disc}}
\renewcommand{\dim}{\operatorname{dim}}

\newcommand{\ord}{\operatorname{ord}}

\newcommand{\Pic}{\operatorname{Pic}}

\newcommand{\Frac}{\operatorname{Frac}}

\newcommand{\Spec}{\operatorname{Spec}}
\newcommand{\MaxSpec}{\operatorname{MaxSpec}}

\newcommand{\pp}{\mathfrak{p}}
\newcommand{\mm}{\mathfrak{m}}

\newcommand{\car}{\operatorname{char}}

\newcommand{\Vol}{\operatorname{Vol}}

\newcommand{\Covol}{\operatorname{Covol}}
\newcommand{\covol}{\operatorname{covol}}

\newcommand{\Disc}{\operatorname{Disc}}
\newcommand{\hatK}{\widehat{K}}
\newcommand{\hatR}{\widehat{R}}
\newcommand{\Ker}{\operatorname{Ker}}

\newcommand{\B}{\mathfrak{B}}
\renewcommand{\div}{\operatorname{div}}
\renewcommand{\P}{\mathcal{P}}
\renewcommand{\deg}{\operatorname{deg}}

\title{Abstract Geometry of Numbers: Linear Forms}

\author{Pete L. Clark}

\address{Department of Mathematics \\ Boyd Graduate Studies Research Center \\ University of Georgia \\ Athens, GA 30602-7403 \\ USA}
\email{pete@math.uga.edu}

\begin{document}
\maketitle

\begin{abstract}
This paper concerns the \textbf{abstract geometry of numbers}: namely the pursuit of certain aspects of geometry of 
numbers over a suitable class of normed domains.  (The standard geometry of 
numbers is then viewed as geometry of numbers over $\Z$ endowed with its standard absolute value.)  In this work we study normed domains of ``linear type'', in which an 
analogue of Minkowski's linear forms theorem holds.  We show that $S$-integer 
rings in number fields and coordinate rings of (nice) affine algebraic curves 
over an arbitrary ground field are of linear type.  The theory is applied to 
quadratic forms in two ways, yielding a Nullstellensatz and a Small Multiple Theorem.
\end{abstract}

\tableofcontents

\section*{Introduction}
\noindent
In a previous paper \cite{ADCI}, the author studied aspects of the theory of quadratic forms over a normed domain $(R,|\cdot|)$.  In particular the notion of a \textbf{Euclidean quadratic form} uses the norm structure in a key way and gives rise to some results of a geometry of numbers (GoN) flavor but in a more abstract algebraic context.  Together with some work applying very elementary GoN to prove representation theorems for integral quadratic forms 
\cite{GoN1}, \cite{GoN2}, \cite{GoN3} this has led me to pursue aspects of 
GoN over normed integral domains: in short an \textbf{abstract GoN}.
\\ \\
The main idea of the present paper is to pursue analogues of Minkowski's Linear Forms Theorem in a normed domain.  In the first part of the paper we develop a theory of normed domains of \textbf{linear type} -- i.e., in which an analogue of Minkowski's Theorem holds.  We show that the domains of most arithmetic interest -- namely, $S$-integer rings in number fields and coordinate rings of affine algebraic curves over an arbitrary ground field -- 
are of linear type.  There is also a quantitative aspect to theory in which we 
ask for the best constant in Minkowski's Theorem, which leads us to define the \textbf{linear constants} $C(R,n)$ of a linear type normed domain.  We provide explicit lower bounds on linear constants for the above kinds of domains.  In simple cases -- $\Z$ and $k[t]$ -- we can show that our lower bounds are sharp, but in most cases the precise determination of the linear constants remains open.  
\\ \\
In the second part of the paper this theory is applied to prove two kinds of 
results for quadratic forms over a normed domain: the Nullstellensatz for isotropic forms and the Small Multiple Theorem for anisotropic forms. \\ \indent 
Given an isotropic quadratic form over a normed domain, it is natural to ask for an upper bound on the size of an isotropic vector 
in terms of the size of the coefficients of the form.  To be sure, there is room for interpretation in the precise meaning of ``the size''.  There are beautiful classical results due in the case $R = \Z$ to J.W.S. Cassels \cite{Cassels55} and in the case $R = k[t]$ to A. Prestel \cite{Prestel87}.  Cassels's result was generalized to $R = \Z_K$ for any number field $K$ by S. Raghavan \cite{Raghavan75}, and Prestel's result was generalized to the coordinate ring of any nonsingular integral affine curve over an arbitrary field by A. Pfister 
\cite{Pfister97}.  Our Nullstellensatz is an abstract version of these theorems: it holds over a suitable linear type normed Dedekind domain In particular we recover the results of Cassels and Prestel but gives variants of the results of Raghavan and Pfister because our measurement of ``the size'' agrees with theirs only when there is a single infinite place.  The Nullstellensatz also applies to $S$-integer rings in number fields, which is a new result.  
\\ \indent
The Small Multiple Theorem holds for certain anisotropic quadratic forms over a suitable linear type normed Dedekind domain.  This is new even over $\Z$, though it has precedent in work of Brauer-Reynolds \cite{Brauer-Reynolds} and Mordell \cite{Mordell51}.  This result opens up an enormous terrain in which one may \emph{try} to apply the --  somewhat mysterious, but usually effective -- computational methods of \cite{GoN1}, \cite{GoN2}, \cite{GoN3}.   
\\ \\
Both instances of ``suitable'' above mean the same thing.  To adapt the arguments over $\Z$ and $k[t]$ one wants the norm to satisfy the triangle inequality, which is unfortunately \emph{not} implied by the formalism of normed domains.  In the case of an $S$-integer ring $R$, the norm satisfies the triangle inequality only when $R = \Z_K$ wih $K = \Q$ or an imaginary quadratic field.  This is a disappointing limitation.  For affine domains, the restricting to one infinite place is considerably less limiting, but it is still not the general case.  Raghavan and Pfister prove results which go beyond these hypotheses, but each of their results involves switching to a different measurement of size: e.g. Raghavan takes as his ``norm'' the maximum of the absolute values at the infinite places: this satisfies the triangle inequality but is only submultiplicative: $|xy| \leq |x||y|$.  There are so many signs that our norm is a natural one: $\S$ 1 of the present work is a rumination on this point -- that the failure of the triangle inequality in so many examples of interest was most distressing. \\ \indent 
Only late in the course of this work did a solution emerge: one can refine the definintion of linear constant so that in both of the above applications one can apply the triangle inequality separately to each of the metric factors of the norm.  This leads to \textbf{multinormed linear constants}.  As someone who had come to regard the triangle inequality as his mortal enemy, it is hard to express how wonderful this refinement seems to me, but I readily admit that it adds an extra layer of complexity.  On a first pass the reader may wish to restrict to the case of one infinite place.
\\ \\
Our approach is a perhaps amusing blend of high and low.  On the algebraic side we work in the context of not necessarily free lattices over an arbitrary Dedekind domain.  This necessitates some background algebra, to which $\S$ 1 is mostly devoted.  However, on the GoN side we work mostly from scratch: in some cases -- e.g. $R = \Z$, $R = \F_q[t]$ -- the Pigeonhole Principle is sufficient.  It is a piece of folklore that the Blichfeldt Lemma (which implies Minkowski's Convex Body Theorem) is a sort of ``Measure Theoretic Pigeonhole Principle.''  We literally give a Measure Theoretic Pigeonhole Principle and use it to deduce an Blichfeldt Theorem 
in a measured group.  This, together with the observation that an $S$-integer ring in a number field is discrete and cocompact in a suitable finite product 
of its completions, is the outer limit of our sophistication: we 
do not (yet!) need reduction theory, adeles, height functions...In the function field case, \emph{in lieu} of using a fully fledged GoN as Mahler, Eichler, and others have developed, we simply invoke the Riemann-Roch Theorem.
\\ \\
I expect that for all Hasse domains and affine domains, the Hermite constants \[ \gamma(n,R) = \sup_{\text{quadratic forms } f \in K[t_1,\ldots,t_n]}  \inf_{x \in (R^n)^{\bullet} }\frac{|f(x)|}{|\disc(f)|^{\frac{1}{n}}}; \] 
should be finite.  This is known for positive forms over totally real number fields \cite{Icaza97}, and the general case is work in progress of J. Hicks and the author.  If $R = \Z$ is any indication, such work should yield a ``Smaller Multiple Theorem'' which is ultimately to be preferred in applications like \cite{GoN1}, \cite{GoN2}, \cite{GoN3}.  However, I have finally followed the advice that it is bad luck to title a paper ``Part I''.  

\section{Normed Dedekind Domains}

\subsection{Elementwise Norms} \textbf{} \\ \\ \noindent 
A \textbf{norm} on a ring $R$ is a function $|\cdot |: R \ra \R^{\geq 0}$ such that \\
(N0) $|x| = 0 \iff x = 0$, \\
(N1) $|x| \geq 1$ for all $x \in R^{\bullet}$; $|x| = 1 \iff x \in R^{\times}$, \\
(N2) $\forall x,y \in R$, $|xy| = |x||y|$.
  \\ \\
A \textbf{normed ring} is a pair $(R,| \cdot |)$ where $| \cdot |$ is a norm on $R$.  A nonzero ring admitting a norm is necessarily a domain.  We denote the fraction field by $K$.  The norm extends uniquely to a homomorphism of groups $(K^{\times},\cdot) \rightarrow (\R^{>0},\cdot)$.

\begin{remark} In \cite{ADCI} our norm functions were required to take values in $\N$.  This is a natural condition but one which is not needed in the present work.
\end{remark}
\noindent
Let $R$ be a domain with fraction field $K$.  We say norms $| \cdot |_1, \cdot | \cdot |_2$ on $R$ are \textbf{equivalent} -- and write $| \cdot |_1 \sim | \cdot |_2$ -- if for all $x \in K$, $|x|_1 < 1 \iff |x|_2 < 1$.  
\\ \\
Elementwise norms are especially easy to understand on a UFD.  Indeed, to 
define an elementwise norm on a UFD one needs to assign to each nonzero principal prime ideal $(\pi)$ of $R$ an integer $a_{\pi} \geq 2$, and any 
such assignment yields an elementwise norm.  In particular a DVR 
carries a unique equivalence class of norms. 
\\ \\
The \textbf{norm group} $\mathcal{N}$ is $|K^{\times}| \subset \R^{> 0}$.  So long as $R \neq K$, its closure $\overline{\mathcal{N}}$ is a nontrivial closed subgroup of $\R^{> 0}$, hence there are just two possibilities: either \\
(i) $\overline{\mathcal{N}} = \R^{> 0}$; we say that $R$ is \textbf{densely normed}, or \\
(ii) $\mathcal{N} \subset q^{\Z}$ for some $q > 1$; we say that $R$ is \textbf{q-normed}.\footnote{Thus $\mathcal{N} = (\tilde{q})^{\Z}$ for some $\tilde{q} = q^a$, $a \in \Z^+$, so the class of $q$-normed rings would have been the same if we had required $\mathcal{N} = q^{\Z}$.  However, we will see that stating it this way is natural for our applications to coordinate rings of affine curves, since an affine algebraic curve over a non-algebraically closed field $k$ need not have any $k$-rational points.} \\ In the $a$-normed case we will find it more convenient to work with \[\deg(\cdot) = \log_q |\cdot |.\] 
The corresponding axioms are: for all $x,y \in R$,
\\ 
(N$_q$0) $\deg x = -\infty$ iff $x = 0$; \\
(N$_q$1) If $x \in R^{\bullet}$, $\deg x \in \N$; $\deg x = 0 \iff x \in R^{\times}$; \\
(N$_q$2) $\forall x,y \in R$, $\deg xy = \deg x + \deg y$.  

\begin{remark}
The function $\deg$ is independent of the equivalence class of the norm.
\end{remark}

\subsection{Ideal norms} \textbf{} \\ \\ \noindent
Let $R$ be a domain.  Then the nonzero ideals of $R$ form a monoid under multiplication, say $\mathcal{I}^+(R)$.  An \textbf{ideal norm} on $R$ is a 
homomorphism of monoids $|\cdot|: \mathcal{I}^+(R) \ra \R^{\geq 1}$ such that 
$|I| = 1 \iff I = R$.  An ideal norm extends to uniquely to a homomorphism 
from the monoid $\mathcal{I}(R)$ of fractional ideals of $R$ to $\R^{> 0}$.  
\\ \\
Ideal norms are especially easy to understand on a Dedekind domain.  Indeed, to 
define an ideal norm on a Dedekind domain one needs to assign to each nonzero prime ideal $\pp$ of $R$ an integer $a_{\pp} \geq 2$, and any 
such assignment yields an ideal norm.  Further, $\mathcal{I}(R)$ is a group 
iff $R$ is Dedekind \cite[Thm. 11.6]{Matsumura}.  
\\ \\
In the present work $R$ will always be a Dedekind domain, and a \emph{normed ring} $(R,|\cdot|)$ means a Dedekind domain 
endowed with an \emph{ideal norm}.

\subsection{Overrings}
\textbf{} \\ \\ \noindent
Let $(R,|\cdot|)$ be a normed Dedekind domain, and let $R'$ be 
an \textbf{overring} of $R$, i.e., a ring intermediate between $R$ and its 
fraction field $K$.  The induced map on spectra $\iota^*: \Spec R' \ra \Spec R$ is an injection, and $R'$ is completely determined by the image $W := \iota^*(\Spec R')$.  Namely \cite[Cor. 6.12]{Larsen-McCarthy}  \[R' = R_W := \bigcap_{\mathfrak{p} \in W} R_{\mathfrak{p}}. \]
This allows us to identify the monoid $\mathcal{I}^+(R_W)$ of ideals of $R_W$ as the free submonoid of the free monoid $\mathcal{I}^+(R)$ on the subset $W$ of $\Spec R$ and thus define an \textbf{overring ideal norm} $| \cdot |_W$ 
on $R_W$ as the composite map 
\[\mathcal{I}^+(R_W) \stackrel{\iota^*}{\ra} \mathcal{I}^+(R) \stackrel{| \cdot |}{\ra} \Z^+. \]  We single out the following properties of $| \cdot |_W$: \\ \\
$\bullet$ Every ideal $I \in \mathcal{R}$ may be uniquely decomposed as 
$W_I I'$ where $W_I$ is divisible by the primes of $W$ and $I'$ is prime to 
$W$, and we have

\[ |I|_W = |W_I I'|_S = |I'|_S = |I'|. \]
$\bullet$ For all ideals $I$, $|I|_W \leq |I|$.  


\subsection{Extended Norms}
\textbf{} \\ \\ \noindent
Let $(R,|\cdot|)$ be a normed Dedekind domain with fraction field $K$.  Let $L/K$ be a finite field extension, and let $S$ be the integral closure of 
$R$ in $L$.  Then $S$ is a Dedekind domain \cite[Thm. 11.7]{Matsumura}.  Let $N_{L/K}: L \ra K$ be the norm in the sense of field theory.  Since $R$ is integrally closed, $N_{L/K}(S) \subset R$.  The composite map 

\[ | \cdot|_S = | \cdot| \circ N_{L/K}: S \ra \R^{> 1} \]
is a norm function on $S$.  We call it the \textbf{extended norm}.  

\subsection{Almost Metric Norms and the Artin Constant}
\textbf{} \\ \\ \noindent
Let $|\cdot|$ be a norm on a ring $R$.  Define 
\[ A(R) = \inf \{ C \in \R^{>0} \ | \ \forall x,y \in R, \ |x+y| \leq C \max (|x|, |y|)\}. \]
If there is no such $C$, then $A(R) = \inf \varnothing = \infty$.  If $A(R) < \infty$ we say that the norm is \textbf{almost metric} and call $A(R)$ the \textbf{Artin constant}.   It follows that for all $x,y \in K$, $|x+y| \leq A(R) \max(|x|,|y|)$, and thus $|\cdot|$ is an absolute value on $K$ in the sense 
of E. Artin.  \\ \indent
When $A(R) = 1$ we say the norm is \textbf{non-Archimedean} or \textbf{ultrametric}.

\begin{lemma}
\label{ARTINLEMMA}
Let $R$ be a domain with fraction field $K$, and let $|\cdot|$ be an 
almost metric norm on $R$ with Artin constant $A(R)$.  \\
a) $A(R) = \max( |1|, |2|)$.  \\
b) For $\alpha \in \R^{>0}$, $A(R,|\cdot|^{\alpha}) = A(R)^{\alpha}$.  \\
c) The map $(x,y) \mapsto |x-y|$ is a metric on $K$ iff $A(R) \leq 2$. \\
d) For $x_1,\ldots,x_n \in K$, $|x_1 + \ldots + x_n| \leq |n| \max_i |x_i|$. 
\end{lemma}
\begin{proof}
For part a), see \cite[p. 16]{Artin}.  Part b) follows immediately.  For part c), see \cite[pp. 4-5]{Artin}.  As for part d): the non-Archmidean case is 
immediate from induction on $|x+y| \leq \max |x|,|y|$.  For the Archimedean 
case: the assertion depends only on the equivalence class of the norm, 
so by scaling we may assume $A(R) =2$.  When the Artin 
constant is $2$, then by Ostrowski's Theorem \cite[p. 24]{Artin}, the absolute value on $K$ 
is obtained by embedding $K$ into $\C$ and restricting the standard Euclidean 
norm.  In particular $|n| = n$ for all $n \in \Z^+$.  Then by induction 
on part c), 
\[ |x_1+ \ldots + x_n| \leq |x_1| + \ldots + |x_n| \leq n \max_i |x_i| = |n| \max_i |x_i|. \]
\end{proof}
\noindent
In particular an almost metric norm is equivalent to a metric norm, and in the sequel we usually renormalize almost metric norms to be metric.

\begin{lemma}
\label{1.4}
Let $|\cdot|$ be a $q$-norm on $R$.  The following are equivalent: \\
(i) $-\deg$ is a discrete valuation on $K$.  \\
(ii) The norm $|\cdot|$ is ultrametric. \\
(iii) The norm $|\cdot|$ is metric. \\
(iv) The norm $|\cdot|$ is almost metric.
\end{lemma}
\begin{proof}
The implications (i) $\iff$ (ii) $\implies$ (iii) $\implies$ (iv) are all immediate.  Assume (iv).  We may adjust the norm within its equivalence class without affecting its ultrametricity, so seeking a contradiction we may suppose that $|\cdot|$ is \emph{not} ultrametric but that it is metric with Artin constant $2$, and thus for all $n \in \Z^+$, $n = |n| = q^{\deg n}$, i.e., 
$\deg n = \log_q n$.  But this implies $\frac{\log_q 3}{\log_q 2} = \log_2 3\in \Q$, a contradiction.  
\end{proof}
\begin{example} Let $R$ be a discrete valuation ring with uniformizer $\pi$.  Then a norm $|\cdot|$ on $R$ is freely determined by mapping $\pi$ to any $q > 1$.  Thus the norms on $R$ lie in a single equivalence class, and they are all $q$-norms.  The equivalent conditions of Lemma \ref{1.4} \emph{do not hold}: we 
have $|1| = |\pi-1| = 1$, but $|\pi| > \max |1|,|\pi-1|$.  
\end{example}

\begin{lemma}
\label{LEMMA3}
Let $m \geq 2$, and let $|\cdot|_1,\ldots,|\cdot|_m$ be inequivalent absolute values on a 
ring $K$.  Suppose that at least one of the following holds: \\
(i) $(K,|\cdot|_2)$ is densely normed.   \\
(ii) There is $\alpha > 1$ such that $\alpha \in \bigcap_{i=1}^m |K^{\times}|_i$.  \\
Define $|\cdot|: K \ra \R$ by $|x| = \prod_{i=1}^m |x|_i$.  
Then $|\cdot|$ is \emph{not} an absolute value on $K$.
\end{lemma}
\begin{proof}
First suppose that (i) holds; let $\alpha > 1$ be an element of $|K^{\times}|_1$.  Fix $n \in \Z^+$.  By Artin-Whaples approximation \cite[p. 9]{Artin} there are $x_n,y_n \in K$ such that 
\[ |x_n|_1 \sim \alpha^n, \ |x_n|_2 \sim \alpha^{-n}, \ |x_n|_k \sim 1 \ \forall k \geq 3. \] 
\[ |y_n|_1 \sim \alpha^{-n}, \ |y_n|_2 \sim \alpha^n, \ |y_n|_k \sim 1 \ \forall k \geq 3. \]
Then $|x_n|, |y_n| \sim 1$, but $\lim_{n \ra \infty} |x_n+y_n| = \infty$, 
so $|\cdot|$ is not an absolute value.  If (ii) holds then the same argument works with $\alpha$ as in the statement of (ii).
\end{proof}

\subsection{Finite Quotient Domains} \textbf{} \\ \\ \noindent
A \textbf{finite quotient domain} is a domain $R$ such that for all nonzero 
ideals $I$ of $R$, $R/I$ is finite \cite{Butts-Wade}, \cite{Chew-Lawn70}, \cite{Levitz-Mott72}.  A finite quotient domain is Noetherian of Krull dimension at most one, hence is Dedekind if and only if 
it is integrally closed.

\begin{prop}
Let $R$ be a finite quotient domain with fraction field $K$. \\
a) Let $L/K$ be a finite extension, and let $S$ be a ring with $R \subset S \subset L$.  Then $S$ is a finite quotient domain.  \\
b) The integral closure $\tilde{R}$ of $R$ in $K$ is a finite quotient domain.  \\
c) The completion of $R$ at a maximal ideal is a finite 
quotient domain. 
\end{prop}
\begin{proof} Part a) is \cite[Thm. 2.3]{Levitz-Mott72}.  In particular, it follows from part a) that $\tilde{R}$ is a finite quotient domain.  That $\tilde{R}$ is a Dedekind ring is part of the Krull-Akizuki Theorem.  Part c) 
follows immediately from part a) and \cite[Cor. 5.3]{Chew-Lawn70}.
\end{proof} 
\noindent
Let $R$ be a finite quotient domain.  For a nonzero ideal $I$ of $R$, we define 
$|I| = \# R/I$.  It is natural to ask 
whether $I \mapsto |I|$ gives an ideal norm on $R$.  

\begin{prop}
\label{2.3}
Let $I$ and $J$ be nonzero ideals of the finite quotient domain $R$.  \\
a) If $I$ and $J$ are comaximal -- i.e., $I + J = R$ -- then $|IJ| = |I||J|$.  \\
b) If $I$ is invertible, then $|IJ| = |I||J|$.  \\
c) The map $I \mapsto |I|$ is an ideal norm on $R$ iff $R$ is integrally closed.
\end{prop}
\begin{proof} Part a) follows immediately from the Chinese Remainder Theorem.  As for part b), we claim that the norm can be computed locally: for 
each $\mathfrak{p} \in \Sigma_R$, let $|I|_{\mathfrak{p}}$ be the norm of 
the ideal $I R_{\mathfrak{p}}$ in the local finite norm domain $R_{\mathfrak{p}}$.   Then 
\begin{equation}
\label{LOCALNORMEQ}
|I| = \prod_{\mathfrak{p}} |I|_{\mathfrak{p}}. 
\end{equation}
To see this, let $I = \bigcap_{i=1}^n \mathfrak{q}_i$ 
be a primary decomposition of $I$, with $\mathfrak{p}_i = \operatorname{rad}(\mathfrak{q}_i)$.  It follows that
$\{\mathfrak{q}_1,\ldots,\mathfrak{q}_n \}$ is a finite set of pairwise 
comaximal ideals, so the Chinese Remainder Theorem applies to give 
\[R/I \cong \prod_{i=1}^n R/\mathfrak{q}_i. \]
Since $R/\mathfrak{q}_i$ is a local ring with maximal ideal corresponding to 
$\mathfrak{p}_i$, it follows that $|\mathfrak{q}_i| = |\mathfrak{q}_i R_{\mathfrak{p}_i}|$, establishing the claim.
\\ 
Using the claim reduces us to the local case, so that we may assume the ideal $I = (x R)$ is principal.  In this case the short exact sequence of $R$-modules
\[0 \ra \frac{xR}{xJ} \ra  \frac{R}{xJ} \ra \frac{R}{(x)J} \ra 0 \]
together with the isomorphism 
\[\frac{R}{J} \stackrel{\cdot x}{\ra} \frac{xR}{xJ} \]
does the job.  \\
c) If $R$ is 
integrally closed (hence Dedekind), every ideal is invertible so 
this is an ideal norm.  The converse is \cite[Thm. 2]{Butts-Wade}.
\end{proof}
\noindent
Thus every finite quotient Dedekind domain comes endowed with an ideal norm: $|I| = \# R/I$.  We call this norm the \textbf{canonical norm}.

\begin{prop}
Let $R$ be a finite quotient Dedekind domain, and let $R_W$ be any overring.  Then the overring norm $|\cdot|_W$ on $R_W$ coincides with the canonical norm.
\end{prop}
\begin{proof}
Let $\iota: R \hookrightarrow R_W$.  It is enough to check $\# R_W/I = 
|I|_W$ for every maximal ideal $I = \P$ of $R_W$.   Using the equality 
of local rings $(R_W)_{\P} = R_{\P \cap R}$ we get 
\[ |\P|_W  = \# R/(\P \cap R) = \# R_{\P \cap R}/(\P \cap R) 
R_{\P \cap R} = \# (R_W)_{\P}/\P (R_W)_{\P} = \# R_W/\P. \]
\end{proof}

\begin{prop}
\label{1.5}
Let $R$ be a finite quotient Dedekind domain with fraction field $K$, 
$L/K$ a separable field extension, and $S$ the integral closure of $R$ in 
$L$.  Then the extended norm $|\cdot|_S$ coincides with the canonical 
norm $|J| = \# S/J$.
\end{prop}
\begin{proof} Put $n = [L:K]$.  Let $\iota: R \hookrightarrow S$ be the inclusion map.  
By multiplicativity, it is enough to treat the case of $J = \mathcal{P}$ a maximal ideal.   Let $\pp = \mathcal{P} \cap R$, and put 
$f = \dim_{R/\pp} S/\P$.
%
%
%
Now recall:
\[ N_{L/K}(\P) = \pp^f. \]
Indeed, when $\pp$ is principal this is \cite[Prop. I.22]{Lang}; since, like 
any ideal in a Dedekind domain, $N_{L/K}(\P)$ can be computed locally, this 
suffices.\footnote{What we have recalled is often taken as 
the \emph{definition} of the norm of an ideal in a finite degree separable field extension.  But our definition applies to the inseparable case as well.}  Thus 
\[ |\P|_S = |N_{L/K}(\pp)| = |\pp^f|  = \# (R/\pp)^f = \# S/\P. \]
\end{proof}
%
%
\noindent
\begin{remark} Let the hypotheses be as in Propsition \ref{1.5} except with $L/K$ purely inseparable.  Then $\mathcal{P}$ is the unique prime of $S$ 
lying over $\pp$ in $R$, so $\iota_* \pp = \P^e$, hence \[ |\P^e|_S = |N_{L/K} \iota_* \pp| = |\pp|^n, \]
so
\[ |\P|_S = |\pp|^{\frac{n}{e}}. \]
It follows that $|\P|_S = \# S/\P$ holds if and only if $ef = n$, i.e., iff $L/K$ is \textbf{defectless} in the sense of 
valuation theory.  If $K$ has transcendence degree one 
over $\F_p$ then every finite extension is defectless \cite[p. 9]{Kuhlmann}, so -- using a simple d\'evissage argument to combine the separable and purely inseparable cases -- the conclusion of Proposition \ref{1.5} also holds when $K$ is a global field of positive characteristic.
\end{remark}

\begin{example} Let $R$ be a DVR with valuation $v$, uniformizing element $\pi$ and residue field $R/(\pi) \cong \F_q$.  The canonical norm on $R$ is $x \in R^{\bullet} \mapsto q^{v(x)}$.  This is the \emph{reciprocal} of the standard 
ultrametric associated to the valuation $v$.  This norm is \emph{not} almost metric: let $x = 1$, $y = \pi^n-1$.  
Then $|x|, |y| = 1$, but $|x+y| = q^n$.  
\end{example}

\subsection{Hasse Domains}
\textbf{} \\ \\ \noindent
Let $K$ be a global field: a finite degree extension of either $\Q$ or $\F_p$.  
A \textbf{place} on $K$ is an equivalence class of almost metric 
norms on $K$.  We denote by $\Sigma_K$ the set of all places of $K$.  Let $S$ 
be a finite, nonempty subset of $\Sigma_K$ containing all the Archimedean places.  We define $\Z_{K,S}$ as the set of all elements $x \in K$ 
such that $|x|_v \leq 1$ for every ultrametric place $|\cdot|_v \in \Sigma_K \setminus S$.  Following O'Meara we call such a ring a \textbf{Hasse domain}.  
Every Hasse domain is a finite quotient Dedekind domain hence comes equipped with the canonical ideal norm $|I| = \# R/I$.
\\ \indent
For the convenience of the reader -- and to fix notation -- we recall some facts.
\\ 
$\bullet$ Suppose $K \cong \Q[t]/(f)$ is a number field.  Then the set of Archimedean places of $K$ is finite and nonempty.  More precisely, if $f$ has $r$ real 
roots and $s$ conjugate pairs of complex roots, then $K$ has $r$ real places -- 
i.e., such that the corresponding completion is isomorphic to the normed field $\R$ -- and $s$ complex places -- i.e., such that the corrsponding completion 
is isomorphic to the normed field $\C$.  We write out the infinite places as $\infty_1,\ldots,\infty_{r+s}$.  The finite places correspond to 
maximal ideals of $\Z_K$, the integral closure of $\Z$ in $K$, which is the 
unique minimal Hasse domain with fraction field $K$: any other Hasse domain $\Z_{K,S}$ with fraction field $K$ is an overring of $R$, obtained as $\bigcap_{\pp \in \MaxSpec R \setminus S_f} R_{\pp}$. \\ 
$\bullet$ Suppose $K$ has characteristic $p > 0$.  Then there is a prime 
power $q = p^f$ such that $K/\F_q(t)$ is a regular extension -- separable, with constant field $\F_q$.  There is a unique smooth, projective geometrically integral curve $C_{/\F_q}$ such that $K = \F_q(C)$ is the field of rational 
functions on $C$.  The places of $K$ are Archimedean and correspond bijectively 
to closed points on $C$, or equivalently to complete $\mathfrak{g}_{\F_q} = \Aut(\overline{\F_q}/\F_q)$-orbits of $\overline{\F_q}$-valued points of $C$.  
Thus the Hasse domains with fraction field $K$ correspond to finite unions of complete $\mathfrak{g}_{\F_q}$-orbits of $\overline{\F_q}$-points of $C$, and any 
such $R$ is the ring of rational functions 
which are regular away from the support of $D$.  There is no unique 
minimal Hasse domain in this case, because we cannot take $D = 0$: the ring 
of functions which are regular on all of $C$ is just $\F_q$.

\begin{prop}
\label{PROP5}
Let $K$ be a number field and let $R = \Z_{K,S}$ be a Hasse domain, endowed with 
its canonical norm $|\cdot|$.  Let $P \in S_f$, and suppose $P$ lies over the 
rational prime $p$.  Let $q_P = |P| = \# R/P$. \\
a) For $x \in K^{\times}$, we have 
\begin{equation}
\label{PROP5EQ}
|x| = \prod_{P \in S_f} 
q_P^{-v_P(x)} \prod_{i=1}^{r+s} |x|_{\infty_i}. 
\end{equation}
b) The norm $|\cdot|$ is almost metric iff $S_f = \varnothing$ and $K= \Q$ or is imaginary quadratic. \\
c) $A(\Z) = 2$.  If $K$ is imaginary quadratic, $A(\Z_K) = 4$.   
\end{prop}
\begin{proof}
a) We recall the \textbf{product formula}: for all $x \in K^{\times}$, 
\[ \prod_{P \in \MaxSpec \Z_K} q_P^{-v_P(x)} \prod_{i=1}^{r+s} |x|_{\infty_i} = 1. \] 
Using this and (\ref{LOCALNORMEQ}) we get 
\[ |x| = \prod_{P \in \MaxSpec R} |x|_P = \prod_{P \in \MaxSpec R} q_P^{v_P(x)} = \prod_{P \in S_f} 
q_P^{-v_P(x)} \prod_{i=1}^{r+s} |x|_{\infty_i}. \]
b) Each factor on the right hand side of (\ref{PROP5EQ}) is an almost metric norm on $K$.  So if there is exactly one factor, $|x|$ is an almost metric norm.  Since there is always at least one 
infinite place, this occurs iff there are no finite places and exactly one 
infinite place, i.e., when $S = S_{\infty}$ and $K = \Q$ or is imaginary quadratic.  By Lemma \ref{LEMMA3}, the norm is not almost metric 
if there is more than one factor on the right hand side of (\ref{PROP5EQ}): hypothesis (i) is satisfied for every Archimedean place.  \\
c) This is immediate from Lemma \ref{ARTINLEMMA}.
\end{proof}
\noindent
\begin{remark} The condition that $S = S_{\infty}$ and $K = \Q$ or imaginary quadratic 
is precisely that of an $S$-integer ring in a number field to have 
finite unit group.  Whenever the unit group is infinite, the set $\{ |u+v| \ | \ u,v \in R^{\times}\}$ is unbounded.  
\end{remark}
\noindent
Proposition \ref{PROP5} has an analogue for Hasse domains of positive characteristic.  In fact it is natural to consider a more general class of 
normed domains, namely coordinate rings of an affine curve over an arbitrary 
ground field.  We do this next.

\subsection{Affine Domains}
\textbf{} \\ \\ \noindent
Let $k$ be a field, let $C_{/k}$ be a smooth, projective geometrically integral 
curve, with fraction field $K = k(C)$.  Let $C^{\circ}$ be an open affine subcurve of $C$ obtained by removing a finite, nonempty set $S_{\infty} = \{\infty_1,\ldots,\infty_m\}$ of closed points of $C$.\footnote{The Galois-theoretic description of divisors in $\S$ 1.7 relied on the perfection of $\F_q$.  This fails for closed points $P \in C$ 
for which the residue field $k(P)$ is an inseparable extension of $k$.} For $1 \leq i \leq m$, let $d_i = [k(P_i):k]$ be the degree 
of $P_i$.  Let
\[ R = k[C^{\circ}] = \bigcap_{P \notin S_{\infty}} R_P \]  
be the ring of all functions regular away from $\infty_1,\ldots,\infty_m$.  Then $R = k[C^{\circ}]$ is a Dedekind domain; we will call such a ring an \textbf{affine domain}.  
\\ \\
The ring $R$ carries a canonical norm up to equivalence: fix $q> 1$.  If $k$ is finite then we take $q = \# k$.  By Zariski's Lemma and the Chinese Remainder Theorem, for all nonzero ideals $I$ of $R$, $R/I$ is a finite-dimensional $k$-vector space, and we put  
\[ |I| = q^{\dim_k R/I}. \]
When $k$ is finite, this is the canonical norm on the Hasse domain $R$.

\begin{prop}
\label{PROP8}
\label{1.12}
a) For $f \in R^{\bullet}$, 
\begin{equation}
\label{PROP8EQ}
|f| = q^{- \sum_{i=1}^m d_i v_{\infty_i}(x)}.
\end{equation}
b) The following are equivalent: \\
(i) $m = 1$.  \\
(ii) $(R,|\cdot|)$ is ultrametric.  \\
(iii) $(R,|\cdot|)$ is almost metric.
\end{prop}
\begin{proof}
The maximal ideals of $R$ are in canonical bijection with the closed points of 
$C^{\circ}$; we use $P$ to denote either one.  Let $f \in R^{\bullet}$; viewing $x$ as a rational function on $C$, consider its 
divisor 
\[ \div f = \sum_{P \in C} \deg P v_P(f) [P]. \]
Exponentiating the relation $\deg \div f = 0$ gives 
\[ q^{\sum_{P \in C^{\circ}} \deg P v_P(f) } =  q^{- \sum_{i=1}^m d_i v_{\infty_i}(x)}. \]
On the other hand, $(f) = \prod_{P \in C^{\circ}} P^{v_P(f)}$, so by the 
Chinese Remainder Theorem 
\[ |f| = q^{\dim_k R/(f)} = q^{\sum_{P \in C^{\circ}} \dim_k R/P^{v_P(f)}} \] \[ = 
q^{\sum_{P \in C^{\circ}} v_P(f) \dim_k R/P} = q^{\sum_{P \in C^{\circ}} \deg P v_P(f)} = q^{- \sum_{i=1}^m d_i v_{\infty_i}(x)}, \]
establishing part a).  As for part b): \\
(i) $\implies$ (ii):  if $m = 1$, then (\ref{PROP8EQ}) shows 
that $|\cdot|$ is obtained by exponentiating the valuation $v_{\infty}$, so 
of course gives an ultrametric.  \\
(ii) $\implies$ (iii) is immediate.  \\
(iii) $\implies$ (i) follows by applying Lemma \ref{LEMMA3} to the absolute 
values $|x|_i = q^{-d_i v_{\infty_i}(x)}$.  

\end{proof}

\subsection{Finite Length Modules, Lattices and Covolumes}
\textbf{} \\ \\ \noindent
Let $R$ be a Dedekind domain with fraction field $K$, and let $M$ be a 
finitely generated $R$-module.  Let $M[\tors]$ be its torsion submodule; 
we have a short exact sequence 
\[ 0 \ra M[\tors] \ra M \ra P \ra 0. \]
The quotient module $P$ is finitely generated and torsionfree over a Dedekind 
domain, hence projective, so the sequence splits: 
\[ M \cong M[\tors] \oplus P. \]
Further, there are maximal ideals $\pp_1,\ldots,\pp_N$ of $R$ and $n_1,\ldots,n_N \in \Z^+$ such that 
\begin{equation}
\label{TORSDEDEQ}
 M[\tors] \cong \bigoplus_{i=1}^N R/\pp_i^{n_i}. 
\end{equation}
The \textbf{length} of $M[\tors]$ is $\sum_{i=1}^N n_i$; an 
$R$-module has finite length if and only if it is finitely generated torsion.  To 
a finite length $R$-module, following \cite[$\S$ I.5]{CL} we attach the invariant 
\[ \chi(M) = \prod_{i=1}^n \pp_i^{n_i}. \]
To see that $\chi(M)$ is well defined we may appeal to the uniqueness properties of the decomposition in (\ref{TORSDEDEQ}) -- which can be easily 
reduced to the corresponding uniqueness statement for torsion modules over a PID -- or observe that $\chi(M)$ is the product of the annihilators of the Jordan-H\"older factors of $M$.  Put $r = \dim_K (P \otimes_R K)$.  Then 
\[ P \cong R^{r-1} \oplus I \]
for a fractional $R$-ideal $I$.  The class of $I$ in $\Pic I$ is an isomorphism 
invariant of $P$.  
\\ \\ 
By an $R$-\textbf{lattice} in $K^n$ we mean a finitely generated $R$-submodule $\Lambda \subset K^n$ 
such that the natural map $\Lambda \otimes_R K \ra K^n$ is a $K$-vector space 
isomorphism.  Since $\Lambda$ is a finitely generated torsionfree module over a 
Dedekind domain, it is projective.  More precisely, the structure theory 
for such modules shows that 
\[ \Lambda \cong R^{n-1} \oplus I \]
where $I$ is a nonzero fractional $R$-ideal.  The class of $I \in \Pic R$ 
is an invariant of $\Lambda$ and indeed classifies $\Lambda$ up to $R$-module 
isomorphism.  Further, the group $\GL_n(K)$ acts on the set of lattices in $K^n$ 
and the orbits are precisely the isomorphism classes of modules, i.e., are 
parameterized by $\Pic R$.  In particular $K^{\times}$ acts on lattices in $K^n$ 
via scalar matrices: for $\alpha \in K^{\times}$, we write $\alpha \Lambda$.  Two lattices which are in the same orbit under this action of scalar matrices 
are \textbf{homothetic}.  
\\ \\
We have the \textbf{standard} $R$-lattice $\mathcal{E}$ in $K^n$: the 
free $R$-module with basis $e_1,\ldots,e_n$.  A lattice $\Lambda$ is 
\textbf{integral} if $\Lambda \subset \mathcal{E}$.  Every lattice is homothetic 
to an integral lattice.  
\\ \\
If $\Lambda_1 \subset \Lambda_2 \subset K^n$ are $R$-lattices, then 
$\Lambda_2/\Lambda_1$ is a finite length $R$-module, so we may define 
$\chi(\Lambda_2/\Lambda_1)$, a nonzero ideal of $R$.  For any 
pair of lattices $\Lambda_1,\Lambda_2$ we define a \emph{fractional} $R$-ideal 
$\chi(\Lambda_2/\Lambda_1)$.  Choose $\alpha \in R^{\bullet}$ such that 
$\alpha \Lambda_1 \subset \Lambda_2$ and put 
\[ a^{-1} \chi(\Lambda_2 / a \Lambda_1). \] 
It is easy to check that this 
is independent of the choice of $a$ (c.f. \cite[$\S$ III.1]{CL}).  Finally, 
we put $\chi(\Lambda) = \chi(\mathcal{E}/\Lambda)$.  
\\ \\
If $|\cdot|$ is an ideal norm on $R$, then for any $R$-lattice $\Lambda$ in $K^n$ we define 
\[ \Covol \Lambda = |\chi(\Lambda)|. \]

\begin{prop}
\label{PROP1.5}
\label{PROP1.8} 
\label{1.13}
Let $\Lambda$ be a lattice in $K^n$, and let $M \in \GL_n(K)$.  \\
a) If $\Lambda = A \mathcal{E}$ is free, then $\chi(\Lambda) = (\det A) R$. \\
b) In the general case we have
\begin{equation}
\label{1.8EQ}
 \Covol (M \cdot \Lambda) = |\det M| \Covol \Lambda. 
\end{equation}
\end{prop}
\begin{proof} 
Equalities of fractional ideals in a Dedekind domain may be checked locally, 
so we immediately reduce to the case of $R$ a DVR. \\
a) For any $a \in R^{\bullet}$ we have $\chi(a \Lambda) = |a|^n \chi(\Lambda)$ and $|\det aA| = |a|^n |\det A|$, so by scaling we may assume that $\Lambda \subset \mathcal{E}$ and thus $A \in M_n(R)$.  Further, we may replace $A$ with $PAQ$ for any $P,Q \in \GL_n(R)$, so we may 
assume that $A$ is in Smith Normal Form: in particular, diagonal.  The
result is immediate in this case.  \\
b) Since $R$ is a DVR, $\Lambda = A \mathcal{E}$ is free and part a) applies:  
Then $\Covol \Lambda = |\det A|$ and $\Covol (M \cdot \Lambda) = 
|\det MA|$; (\ref{1.8EQ}) follows.
\end{proof}
\noindent
Let $k$ be a field, $C_{/k}$ a smooth, geometrically integral projective curve, 
and $\infty_1,\ldots,\infty_m$ closed point of $C$ of degrees $d_1,\ldots,d_m$.  Let $C^{\circ} = C \setminus \{\infty_1,\ldots,\infty_m\}$ and $R = k[C^{\circ}]$.  As in $\S 1.8$, we fix $q > 1$ and endow $R$ with the ideal $q$-norm $I \mapsto |I| = q^{\dim_k R/I}$.
\begin{lemma}
\label{LINALGLEMMA1}
\label{1.14}
For any integral lattice $\Lambda \subset R^n$, we have 
\[ \Covol \Lambda = q^{\dim_k R^n/\Lambda}. \]
\end{lemma}
\begin{proof} 
Let $\Lambda = \Lambda_0 \subset \Lambda_1 \subset \ldots \subset \Lambda_N = R^n$ be 
a maximal strictly ascending chain of $R$-submodules, so that $\Lambda_{i-1}/\Lambda_i \cong R/\pp_i$ for some maximal ideal $\pp_i$ of $R$.  
Then 
\[ \Covol \Lambda = |\prod_{i=1}^N \pp_i| = \prod_{i=1}^N |\pp_i| =  \prod_{i=1}^N q^{\dim_k R/\pp_i} \] \[ = 
q^{\sum_{i=1}^N \dim_k R/\pp_i} = q^{\sum_{i=1}^N \dim_k \Lambda_i/\Lambda_{i-1}} = q^{\dim_k R^n/\Lambda}. \] 
\end{proof}

\section{Linear Type Domains}

\subsection{Basic Definitions}
\textbf{} \\ \\ \noindent
Let $(R,|\cdot|)$ be an ideal normed Dedekind domain with norm group $\mathcal{N}$ and fraction field $K$.  
We say that $R$ is of \textbf{linear type} if for all $n \in \Z^+$ there is $C > 0$ such that: for all $M = (m_{ij}) \in \GL_n(K)$, 
an $R$-lattice $\Lambda \subset K^n$ and $\epsilon_1,\ldots,\epsilon_n \in \overline{\mathcal{N}}$ such that 
\begin{equation}
\label{MAINLINEARINEQ}
| \det M| \Covol \Lambda \leq C \prod_{i=1}^n \epsilon_i, 
\end{equation}
there is $x = (x_1,\ldots,x_n) \in \Lambda^{\bullet}$ such that 
\[ \forall 1 \leq i \leq n, \ \left|\sum_{j=1}^n m_{ij} x_j\right| \leq \epsilon_i. \]
When $R$ is of linear type, we let $C(R,n)$ be the supremum over all $C \in \overline{\mathcal{N}}$ such that (\ref{MAINLINEARINEQ}) holds.  We call the $C(R,n)$'s the \textbf{linear constants} of 
$R$. 

\begin{remark}
If $|\cdot|_1$ and $|\cdot|_2$ are equivalent norms on $R$, then if one is of 
linear type then so is the other.  If $|\cdot|_2 = |\cdot|_1^{\alpha}$, 
then $C_2(R,n) = C_1(R,n)^{\alpha}$.  
\end{remark}
\noindent
When $(R,\cdot)$ is $q$-normed of linear type, then $C(R,n) \in q^{\Z}$, so it 
is convenient to put $c(R,n) = \log_q C(R,n) \in \Z$.  

\begin{prop}
\label{LINTYPEUPPERBOUND}
\label{2.1}
a) If $(R,|\cdot|)$ is densely normed of linear type, then for all $n \in \Z^+$, 
\[ C(R,n) \leq 1. \]
b) Let $(R,|\cdot|)$ is $q$-normed, of linear type, and let $\alpha \in R^{\bullet} \setminus R^{\times}$.  Then for all $n \in \Z^+$,  
\[ c(q,n) \leq (\deg \alpha)n-1. \]
\end{prop}
\begin{proof}
a) Fix $n \in \Z^+$, take $M = 1$, $\Lambda = R^n$ and $e_1 = \ldots = e_n= (1-\delta)$ for some $0 < \delta < 1$.  There is no nonzero 
$x = (x_1,\ldots,x_n) \in R^n$ such that $|x_i| \leq (1-\delta)$ for all 
$1 \leq i \leq n$, so for any $C > 0$ satisfying the linear 
type condition we have $1 > C(1-\delta)^n$ or $C <(1-\delta)^{-n}$.  Since 
this holds for all $\delta > 0$, we get $C(R,n) \leq 1$.  \\
b) Fix $n \in \Z^+$, take $M = 1$, $\Lambda = R^n$ and $e_1 = \ldots = e_n = |x|^{-1}$.  There is no nonzero $x = (x_1,\ldots,x_n) \in R^n$ such that 
$\deg x_i \leq e_i$ for all $1 \leq i \leq n$, so 
\[ 0 = \deg (\det 1) \covol R^n > c(R,n) + \sum_{i=1}^n -\deg \alpha. \]
Thus $c(R,n) < (\deg \alpha) n$.  Since $c(R,n) \in \Z$, we have $c(R,n) \leq (\deg \alpha) n-1$.   
\end{proof}

    
\begin{example} \textbf{Minkowski's Linear Forms Theorem} implies that $\Z$ is of linear type with $C(\Z,n) \leq 1$ for all $n \in \Z^+$.  We will give a(n even) more elementary proof in $\S$ 3.2 using no more than the Pigeonhole Principle.  Together with Proposition \ref{LINTYPEUPPERBOUND}a) this gives $C(\Z,n) = 1$ for all $n \in \Z^+$ and also that this upper bound is sharp.  There is a \emph{boundary case}  left open by our setup: is $1$ an acceptable choice of $C$ in the linear type condition?  We will see that it is and actually prove a little more.
\end{example} 

\begin{example}\textbf{Tornheim's Linear Forms Theorem} implies that for any 
field $k$ and $q > 1$, $k[t]$ with norm $|f| = q^{\deg f}$ 
is $q$-normed and of linear type with $c(k[t],n) \geq n-1$.  Together with Proposition \ref{LINTYPEUPPERBOUND}b) applied with $\alpha = t$ this gives $c(k[t],n) = n-1$ for 
all $n \geq 1$ and also that this upper bound is sharp. 
\end{example}  
\noindent
It turns out to be useful to compare the linear type condition with the following \emph{a priori} weaker 
one: an ideal normed Dedekind domain $(R,|\cdot|)$ with fraction field $K$ is of 
\textbf{linear congruential type} if for all $n \in \Z^+$ there is 
$C' \in \overline{\mathcal{N}}$ such that: for all integral lattices $\Lambda \subset K^n$ and 
$\epsilon_1,\ldots,\epsilon_n > 0$ such that 
\begin{equation}
\label{MAINLINCONGEQ}
\Covol \Lambda \leq C' \prod_{i=1}^n \epsilon_i, 
\end{equation}
there is $x = (x_1,\ldots,x_n) \in \Lambda^{\bullet}$ such that for 
all $1 \leq i \leq n$, $|x_i| \leq \epsilon_i$.  If $R$ is of linear congruential type, for $n \in \Z^+$ we let $C'(R,n)$ be the supremum over constants $C'$.  We call the $C'(R,n)$'s the \textbf{linear congruential constants} of $R$.

\begin{prop}
\label{EQUIVALENCEPROP}
\label{2.4}
A normed domain is of linear type iff it is of linear congruential 
type.  Further, for all $n \in \Z^+$ we have $C(R,n) = C'(R,n)$.  \\
\end{prop}
\begin{proof}
Step 0: It is clear that linear type implies linear congruential type and 
that $C'(R,n) \leq C(R,n)$ for all $n \in \Z^+$.  \\
Step 1: Suppose $R$ is of linear congruential type.  Let $\Lambda$ be \emph{any} $R$-lattice in $K^n$ and $\epsilon_1,\ldots,\epsilon_n \in \overline{\mathcal{N}}$ be such that 
\[ \Covol \Lambda \leq C'(R,n) \prod_{i=1}^n \epsilon_i. \]
We claim that there is $y \in \Lambda^{\bullet}$ with $|y_i| \leq C'(R,n) \epsilon_i$ for all $1 \leq i \leq n$.  Choose $a \in R^{\bullet}$ such that $a \Lambda \subset \mathcal{E}$ and $|a| \epsilon_i \in |R^{\bullet}|$ for all $i$.   Then 
\[ \Covol a \Lambda = |a|^n \Covol \Lambda \leq C'(R,n) \prod_{i=1}^n |a| \epsilon_i, \]
so there is $x = (x_1,\ldots,x_n) \in (a \Lambda)^{\bullet}$ with $|x_i| \leq |a| \epsilon_i$ for all $1 \leq i \leq n$.  
Put $y = \frac{1}{a} x \in \Lambda^{\bullet}$.  Then $|y_i| \leq \epsilon_i$ 
for all $1 \leq i \leq n$. \\
Step 2: Let $M \in \GL_n(K)$, $\Lambda \subset K^n$ be an $R$-lattice, 
and $\epsilon_1,\ldots,\epsilon_n \in \overline{\mathcal{N}}$ such that 
\[ |\det M| \Covol \Lambda \leq C'(R,n) \prod_{i=1}^n \epsilon_i. \]
Put $\Lambda_M = M \Lambda$.  Suppose (\ref{MAINLINEARINEQ}) holds.  Then by 
Proposition \ref{PROP1.5}, 
\[ |\det M| \Covol \Lambda = \Covol \Lambda_M \leq C(R,n) \prod_{i=1}^n \epsilon_i, \]
so by the assumed special case there is $y = (y_1,\ldots,y_n)^{\bullet} \in 
\Lambda_M = M \Lambda$ such that for all $1 \leq i \leq n$, $|y_i| \leq 
\epsilon_i$.  But for $1 \leq i \leq n$, $y_i = \sum_{j=1}^n m_{ij} x_j$ 
for $x_j \in R$, so there is $x \in \mathcal{E}^{\bullet}$ such that for all 
$1 \leq i \leq n$, $\left| \sum_{j=1}^n m_{ij} x_j \right| \leq \epsilon_i$.  
\end{proof}

\begin{remark} 
\label{2.6}
Consider the special case of the linear type condition 
in which $\Lambda = \mathcal{E}$.  A linear change of variables shows that this is equivalent to the linear type condition for all \emph{free lattices}, hence 
to the full linear type condition when $R$ is a PID.  
\end{remark}

\subsection{Overrings}

\begin{prop}
\label{LINEAROVERPROP}
\label{2.7}
Let $(R,|\cdot|)$ be a normed Dedekind domain.  Let $W \subset \MaxSpec R$, and let $R_W$ be the corresponding overring, endowed with the norm of $\S 1.5$.  \\
a) For all $n \in \Z^+$, we have $C(R_W,n) \leq C(R,n)$.  \\
b) In particular, if $R$ is of linear type, so is $R_W$.
\end{prop}
\begin{proof} By Proposition \ref{EQUIVALENCEPROP} we may deal with linear \emph{congruential} type and the constants $C'(R,n)$, $C'(R_W,n)$ instead. 
Now everything works out easily: first, every integral $R_W$-lattice $\Lambda$ is of the form $L \otimes_R R_W$ for some $R$-lattice $L$ such that $\chi(L)$ is not divisible by any 
prime in $W$, and thus $\Covol \Lambda = \Covol L$.  Thus, if 
$\epsilon_1,\ldots,\epsilon_n \in \overline{\mathcal{N}}$ are such that 
\[ \Covol \Lambda = \Covol L \leq c(R,n) \prod_{i=1}^n \epsilon_i, \]
then there is $x \in L^{\bullet}$ such that $|x_i| \leq \epsilon_i$ for all 
$1 \leq i \leq n$.  Then $x \in \Lambda^{\bullet}$ and $|x_i|_W \leq 
|x_i| \leq \epsilon_i$ for all $1 \leq i \leq n$.
\end{proof}


\subsection{Extended Norms}

\begin{ques}
Let $R$ be a linear type normed Dedekind domain with fraction field $K$, $L/K$ a finite field extension, and $S$ the integral closure of $R$ in $L$, endowed with its canonical norm of $\S 1.6$. Must $S$ be of linear type?
\end{ques}
\noindent
Although the question is a natural one, we are not able to give any kind of 
answer in this abstract setting.  The problem is that when we convert a system of inequalities $|\sum_{j=1}^n m_{ij} x_j|_S \leq \epsilon_i$ over $S$ to a system of inequalties over $R$ -- namely 
\[ \left|N_{L/K}\left(\sum_{j=1}^n m_{ij} x_j \right)\right|_R \leq \epsilon_i, \]
then the new system of inequalities is now not of a linear nature. 



\subsection{Multinormed Linear Constants}
\textbf{} \\ \\ \noindent
We give here a refinement of the notion of linear constant which takes into account that in the examples of interest to us, the norm $|\cdot|$ on $R$ need not be almost metric but is \textbf{multimetric}: a finite product of almost metric norms.  Note in particular that the canonical norms on every Hasse domain and affine domain are multimetric.  
\\ \\
We say an ideal normed Dedekind domain $(R,|\cdot|)$ is \textbf{multinormed} 
if there are elementwise norms $|\cdot|_1,\ldots,|\cdot|_m$ on $R$ such that 
$|x| = \prod_{j=1}^m |x|_j$ for all $1 \leq j \leq m$.  We say that $(R,|\cdot|)$ is \textbf{multimetric} if each norm $|\cdot|_j$ is almost metric.  (This is is the case of interest to us.)  For $1 \leq j \leq m$ we put $\mathcal{N}_j = |K^{\times}|_j$. \\ \indent
The norm $|\cdot|$ is of \textbf{q-type} iff there is $q > 0$ such that $\mathcal{N}_j \subset q^{\Z}$ for all $j$: this is the situation for affine domains.  We emphasize that more than one choice of $q$ is always possible but that such a choice will always be given as part of the structure.  As in the $m = 1$ case we put $\deg_j = \log_q |\cdot|_j$.  When each $-\deg j$ is a discrete valuation, we say the norm is \textbf{totally ultrametric}.  \\ \indent 
The norm is \textbf{totally dense} if $\mathcal{N}_j$ is dense for each $j$.  If each $|\cdot|_j$ is metric, this is equivalent to each $|\cdot|_j$ being Archimedean, and we use the terminology \textbf{totally Archimedean}.  The canonical norm on $R= \Z_K$, $K$ a number field, is totally 
Archimedean.  The norm is of \textbf{mixed type} if some $\mathcal{N}_j$ is dense and some 
$\mathcal{N}_{j'}$ is not.  The canonical 
norm on $R = \Z_{K,S}$ when $S \neq \varnothing$ is of mixed multimetric type. 
\\ \\
A multimetric ideal normed Dedekind domain $R$ is of \textbf{multinormed linear type} if for all $n \in \Z^+$ there is $C \in \overline{\mathcal{N}}$ such that: given $M = (m_{ij}) \in \GL_n(K)$, 
an $R$-lattice $\Lambda \subset K^n$ and for all $1 \leq j \leq m$ constants 
$\epsilon_{1j},\ldots,\epsilon_{nj} \in \overline{\mathcal{N}_j}$
 such that 
\begin{equation}
| \det M| \Covol \Lambda \leq C \prod_{i,j} \epsilon_{ij}, 
\end{equation}
there is $x = (x_1,\ldots,x_n) \in \Lambda^{\bullet}$ such that 
\[ \forall i,j, \ \left|\sum_{k=1}^n m_{ik} x_k\right|_j \leq \epsilon_{ij}. \]
When $R$ is of multinormed linear type, we let $C_M(R,n)$ be the supremum over all $C \in \overline{\mathcal{N}}$ such that (\ref{MAINLINEARINEQ}) holds.  We call the $C_M(R,n)$'s the \textbf{multinormed linear constants} of 
$R$.   
\noindent
We can now introduce the notion of \textbf{multinormed linear congruential type} 
and the associated constants $C'_M(R,n)$.  Just as in the linear type case it turns out that multimetric linear congruential type is equivalent to multimetric linear type and $C'_M(R,n) = C_M(R,n)$ for all $n$.  In the sequel we will estimate the multinormed linear constants using this equivalence.

\subsection{Diophantine Approximation}
\textbf{} \\ \\ \noindent
One of the most basic and important applications of Minkowski's Linear Forms Theorem is to Diophantine Approximation.  The formalism of domains of linear 
type and $q$-linear type yields analogues of these classical results.

\begin{thm}
\label{2.10}
Let $(R,|\cdot|)$ be a multinormed linear type Dedekind domain.   Let $n \in \Z^+$, $M \in \overline{\mathcal{N}} \cap (1,\infty)$, $\theta_1,\ldots,\theta_n \in K$.   \\
a) Suppose $R$ is densely normed.  If there is $C > 0$ with \[|M|^{-1} < C <  C_M(R,n+1), \] then there are $x_1,\ldots,x_n \in R$ and $x_{n+1} \in R^{\bullet}$ 
such that \\
$\bullet$  $\forall 1 \leq j \leq m$, $\forall 1 \leq i \leq n,$ $|x_{n+1} \theta_i - x_i|_j \leq (C |M|)^{\frac{-1}{mn}}$, and \\
$\bullet$ $\forall 1 \leq j \leq m$, $|x_{n+1}|_j < |M|^{\frac{1}{m}}$.  \\
b) Suppose $R$ is $q$-normed.  If 
\[ m(n+1) - \deg M \leq c_M(R,n+1), \]
then there are $x_1,\ldots,x_n \in R$, $x_{n+1} \in R^{\bullet}$ such that both of the following hold:\\
$\bullet$ For all $1 \leq i \leq n, \ 1 \leq j \leq m$, $\deg_j(x_{n+1}\theta_i - x_i) \leq \frac{m(n+1)-1-c_M(R,n+1) - \deg M}{mn}$,  \\
$\bullet$ For all $1 \leq j \leq m$, $\deg_j x_{n+1} \leq \deg_j M - 1$.
\\
c) Suppose $R$ is of mixed multinormed linear type, with $\mathcal{N}_j$ 
dense for $1 \leq j \leq m'$ and $q_j$-normed for $m' +1 \leq j \leq m$.   
If there is $C > 0$ with \[|M|^{-1} < C <  C_M(R,n+1), \] then there are $x_1,\ldots,x_n \in R$ and $x_{n+1} \in R^{\bullet}$ 
such that \\
$\bullet$  $\forall 1 \leq j \leq m'$, $\forall 1 \leq i \leq n,$ $|x_{n+1} \theta_i - x_i|_j \leq (C |M|_j)^{\frac{-1}{n}}$, \\
$\bullet$ $\forall m' < j \leq m$, $\forall 1 \leq i \leq n$, 
$|x_{n+1} \theta_i - x_i|_j \leq 1$, \\
$\bullet$ $\forall 1 \leq j \leq m'$, $|x_{n+1}|_j < |M|_j$, and \\  
$\bullet$ $\forall m' < j \leq m$, $|x_{n+1}|_j \leq 1$.  
\end{thm} 
\begin{proof}
In all cases we take $\Lambda = R^{n+1}$ and
\begin{equation}
\label{2.10EQ}
A = \left[ \begin{array}{ccccccc} -1 & 0 & \ldots & 0 & \theta_1 \\
0 & -1 & \ldots & 0 & \theta_2 \\
\vdots & & & & \vdots \\
0 & 0 & \ldots & -1 & \theta_n \\
0 & 0 &\ldots & 0 & 1 
\end{array} \right], 
\end{equation}
so $|\det A|_j = 1$ for all $j$.  \\
a) For all $1 \leq i \leq n$ and all $1 \leq j \leq m$, put $\epsilon_{ij} = (C|M|)^{\frac{-1}{nm}}$; for all $1 \leq j \leq m$, put $\epsilon_{(n+1) j} = |M|^{\frac{1}{m}} -\delta$ for some $\delta > 0$.  Then for sufficiently small $\delta$, $C_M(R,n) \prod_{i,j} \epsilon_{ij} > 
|\det A| \Covol \Lambda$ and thus there is a nonzero $x = (x_1,\ldots,x_{n+1}) \in R^{n+1}$ such that $|x_{n+1} \theta_i - x_i|_j \leq  (C |M|)^{\frac{-1}{nm}}$ for all $1 \leq i \leq n$, $1 \leq j \leq m$ and $|x_{n+1}|_j < |M|_j$ for all $1 \leq j \leq m$.  If 
$x_{n+1} = 0$ then we would have $|x_i| \leq (C |M|)^{\frac{-1}{n}} < 1$, so 
$x_1 = \ldots = x_n = 0$ and thus $x = 0$, contradiction.  \\
b) For $1 \leq i \leq n$ and $1 \leq j \leq m$, put 
\[ e_{ij} = \lceil \frac{m-c_M(R,n+1) - \deg M}{mn} \rceil \] \[ \leq
\frac{m-c_M(R,n+1) - \deg M}{mn} + \frac{mn-1}{mn} \] \[ = \frac{m(n+1)-1-c_M(R,n+1) - \deg M}{mn}, \]
and for all $1 \leq j \leq m$,
\[ e_{(n+1)j} = (\deg_j M) -1. \]
Then 
\[ c_M(R,n+1) + \sum_{1 \leq i \leq n+1, 1 \leq j \leq m} e_{ij} \geq  0 = \deg (\det A) + \covol \Lambda,\]             
so by definition of $c_M(R,n+1)$ there is $x = (x_1,\ldots,x_n,x_{n+1}) \in (R^{n+1})^{\bullet}$, not all zero, such that 
\[\forall 1 \leq i \leq n, \forall 1 \leq j \leq m, \  \deg_j (x_{n+1} \theta_i - x_i) 
\leq  \lceil \frac{m-c_M(R,n+1) - \deg M}{mn} \rceil \]
\[ \leq 
\frac{m(n+1)-1-c_M(R,n+1) - \deg M}{mn} < 0, \]
by our hypothesis, so for all $1 \leq i \leq n$,
\[ \deg x_{n+1} \theta_i- x_i < 0. \]
As above, $x_{n+1} \neq 0$: otherwise $x_1 = \ldots = x_n = 0$ and thus $x = 0$, contradiction.  \\
c) This is very similar to part a) and may be left to the reader.
\end{proof}

\section{The Group Theoretic Pigeonhole Principle}

\subsection{The Group Theoretic Pigeonhole Principle} 

\begin{thm}(Group Theoretic Pigeonhole Principle)
\label{LINPHP}
\label{LPHP}
\label{GTPP}
Let $G$ be a group -- not necessarily commutative, but written additively -- 
and let $\Lambda$ be a subgroup of $G$.  Let $S \subset G$, and let $D(S) = \{ s_1 - s_2 \ | \ s_1, s_2 \in S\}$.  If for a cardinal number $\kappa$ we have 
\[ \# S >  \kappa \cdot \# [G:\Lambda], \]
then there are at least $\kappa$ nonzero elements of $D(S) \cap \Lambda$.  
\end{thm}
\begin{proof}
Let $G/\Lambda$ be the set of right cosets of $\Lambda$ in $G$, and let $\Phi: G \ra G/\Lambda$ be the map $g \in G \mapsto \Lambda + g$.  If $\# (\Phi^{-1}(y) \cap S) \leq \kappa$ for all $y \in G/H$ then $\# S \leq 
\kappa \cdot \# [G:\Lambda]$: contradiction.  So there is $S' \subset S$ with $\# S' > \kappa$ and $\Phi(s_1) = \Phi(s_2)$ for all $s_1,s_2 \in S'$. Fix $s_0 \in S'$ and put $S'' = S' \setminus \{s_0\}$, so $\# S'' \geq \kappa$.  As $s$ runs through $S''$, $s-s_0$ are distinct nonzero elements of $D(S) \cap \Lambda$.
\end{proof}

\subsection{The Classical Case}

\begin{thm}
\label{ZLINEARTYPETHM}
\label{3.2}
For all $n \in \Z^+$, $C(\Z,n) = 1$.
\end{thm}
\begin{proof}
Let $\Lambda \subset \Z^n$ be a lattice and $\epsilon_1,\ldots,\epsilon_n \in \R^{> 0}$ with \[[\Z^n:\Lambda] = \Covol \Lambda \leq \prod_{i=1}^n \epsilon_i. \]  Let $G = \Z^n$.  Put \[S = \Z^n \cap \prod_{i=1}^n [0,\epsilon_i]. \]
Then 
\begin{equation}
\label{THM10EQ1}
 \# S = \prod_{i=1}^n \left( \lfloor \epsilon_i \rfloor + 1 \right) > \prod_{i=1}^n \epsilon_i \geq \# G_2, 
\end{equation}
so by Theorem \ref{LINPHP} there are $s_1 \neq s_2 \in S$ with $s_1-s_2 \in \Lambda$.  Then $x = (x_1,\ldots,x_n) = s_1 - s_2 \in \Lambda^{\bullet}$ and has $|x_i| \leq \epsilon_i$ 
for all $1 \leq i \leq n$.  It follows that $C(\Z,n) \geq 1$.  Combining this with Proposition \ref{LINTYPEUPPERBOUND} gives $C(\Z,n) = 1$.
\end{proof}
\noindent
In (\ref{THM10EQ1}) we have $\lfloor \epsilon_i \rfloor + 1 > \epsilon_i$ for all $i$ and thus $\prod_{i=1}^n \left( \lfloor \epsilon_i \rfloor + 1 \right) > \prod_{i=1}^n \epsilon_i$.  This is more 
than we need: it would be enough to have $n$ inequalities any one of 
which is strict.  Using this one obtains the following mild strengthening 
of Theorem \ref{ZLINEARTYPETHM}.

\begin{thm}
\label{ZLINTYPE2}
\label{3.3}
Let $\Lambda \subset \Q^n$ be a $\Z$-lattice, and let $\epsilon_1,\ldots,\epsilon_n > 0$ be such that $\Covol \Lambda \leq \prod_{i=1}^n \epsilon_i$.  Fix an 
index $i_{\bullet} \in \{1,\ldots,n\}$.  Then there is $x \in \Lambda^{\bullet}$ 
such that $|x_i| < \epsilon_i$ for all $i \neq i_{\bullet}$ and $|x_{i_{\bullet}}| \leq \epsilon_{i_{\bullet}}$.
\end{thm}
\begin{proof}
As in the proof of Proposition \ref{2.4}, it is no loss of generality 
to assume that $\Lambda \subset \mathcal{E}$ is an integral lattice.  The proof  is the same as above except we take
\[ S = \Z^n \cap \left( [0,\epsilon_{i_{\bullet}}] \times \prod_{i \neq i_{\bullet}} 
[0,\epsilon_i) \right).\]
Then $\# (\Z \cap [0,\epsilon_i)) \geq \epsilon_i$ and $\# (\Z \cap [0,\epsilon_{i_{\bullet}}]) = \lfloor \epsilon_{i_{\bullet}} \rfloor + 1 > \epsilon_{i_{\bullet}}$, so $\# S > \prod_{i=1}^n \epsilon_i$.  
\end{proof}

\begin{cor}
\label{3.4}
Let $n \in \Z^+$, $M > 1$, $\theta_1,\ldots,\theta_n \in \R$.  There are 
$x_1,\ldots,x_{n+1} \in \Z$ with 
\[ \forall 1 \leq i \leq n, \ |x_n \theta_i - x_i| \leq M^{\frac{-1}{n}}, \]
\[ 0 < |x_{n+1}| < M. \]
\end{cor}
\begin{proof}
The argument is the same as the proof of Theorem \ref{2.4}a) except using the slight strengthening of $C(\Z,n) = 1$ afforded by Theorem \ref{3.3}.
\end{proof}

\begin{cor}
\label{LINFORMCOR}
Let $m,n \in \Z^+$, $d_1,\ldots,d_m \in \Z^+$, $\epsilon_1,\ldots,
\epsilon_n \in \R^{> 0}$, and suppose 
\begin{equation}
\label{LINFORMEQ1}
\prod_{j=1}^n \epsilon_j \geq \prod_{i=1}^m d_i.
\end{equation}  Let $A = (a_{ij}) \in M_{m,n}(\Z)$ and $j_0 \in \{1,\ldots,n\}$.  
Fix $j_{\bullet} \in \{1,\ldots,n\}$.  There is $(x_1,\ldots,x_n) \in (\Z^n)^{\bullet}$ with \\
(i) $\sum_{j=1}^n a_{ij} x_j \equiv 0 \pmod{d_i}$ for $1 \leq i \leq m$ and \\
(ii) $|x_{j_{\bullet}}| \leq \epsilon_{i_{\bullet}}$, and $|x_i| < \epsilon_i$ for all $i \neq i_{\bullet}$.  
\end{cor}
\begin{proof}
For each $1 \leq i \leq m$, the set $\Lambda_i = \{x \in \Z^n \ | \ 
\sum_{j=1}^n a_{ij} x_j \equiv 0 \pmod{d_i} \}$ is a sublattice of $\Z^n$ of 
index at most $d_i$.  Therefore $\Lambda = \bigcap_{i=1}^m \Lambda_i$ is 
a sublattice of $\Z^n$ of index at most $\prod_{i=1}^m d_i$.  Now apply 
Theorem \ref{ZLINTYPE2}.
\end{proof}
\noindent
Various special cases of Corollary \ref{LINFORMCOR} have appeared 
in the literature.  The case $m = 1$, $n =2$, $\epsilon_1 = \epsilon_2$ 
is due to A. Thue \cite{Thue}; the case $m = 1$, $n =2$ is due to I.M. Vinogradov \cite{Vinogradov}.  The case of $m,n$ 
arbitrary, but all $d_i$'s and $\epsilon_j$'s equal is due to Brauer-Reynolds 
\cite{Brauer-Reynolds}.  
The general case -- but with strict inequalities in both the hypothesis and 
conclusion is due to Stevens-Kuty \cite{Stevens-Kuty}.  Most of all, 
the result with arbitrary $m$ and $n = 3$ is due to Mordell \cite[p. 325]{Mordell51}.  


\subsection{A Linear Type Criterion for Finite Quotient Domains}

\begin{thm}
\label{METRICLINTYPETHM}
\label{3.5}
Let $R$ be an almost metric finite quotient Dedekind domain, endowed with its canonical norm. \\
a) Suppose $R$ is densely normed and that there are $\kappa,E > 0$ such that for $e \geq E$, $\# \{ x \in R \ | \ |x| \leq e\} \geq \kappa e$.  Then 
\[ C(R,n) \geq \left( \frac{\kappa}{A(R)} \right)^n. \]
b) Suppose $R$ is $q$-normed and that there are $k \in \Z$, $A \in \N$ such that for all integers $a \geq A$, 
$\# \{ x \in R \ | \ \deg x \leq a \} \geq k+a$.  Then 
\[ c(R,n) \geq nk-1. \]
\end{thm}
\begin{proof}
a) Let $C < \left( \frac{\kappa}{A(R)} \right)^n$, let $\Lambda$ be an integral 
$R$-lattice in $K^n$, and let $\epsilon_1,\ldots,\epsilon_n \in \R^{>0}$ be such that 
$\Covol \Lambda \leq C \prod_{i=1}^n \epsilon_i$.  Suppose first that $\epsilon_i \geq E A(R)$ for all $i$.  
Let $S = \{x \in R^n \ | \ |x_i| \leq \frac{\epsilon_i}{A(R)} \ \forall i\}$.
By definition of $\kappa$ and $E$ we have 
\[ \# S \geq \prod_{i=1}^n \kappa \frac{\epsilon_i}{A(R)} > C \prod_{i=1}^n \epsilon_i \geq \Covol \Lambda, \]
so by Theorem \ref{LPHP} there are $s \neq s' \in S$ 
such that $x = s-s' \in \Lambda$.  Then for all $1 \leq i \leq n$, 
$|x_i| = |s_i - s_i'| \leq A(R) \max |s_i|, |s_i|' \leq \epsilon_i$, so we 
have verified the linear congruential type condition in this case.  \\ \indent
  \\
Now choose $a \in R^{\bullet}$ such that for all $1 \leq i \leq n$, 
$|a| \epsilon_i \geq E A(R)$.  Then 
\[ \Covol a \Lambda = |a|^n \Covol \Lambda \leq C \prod_{i=1}^n |a| \epsilon_i, \]
so by the case done above there is $y \in (a\Lambda)^{\bullet}$ with $|y_i| \leq |a| \epsilon_i$ for all $1 \leq i \leq n$.  Put $x = \frac{1}{a} y$, so $x \in \Lambda^{\bullet}$ 
and $|x_i| \leq \epsilon_i$ for all $1 \leq i \leq n$.  It follows that $C(R,n) \geq \left( \frac{\kappa}{A(R)} \right)^n$.  \\
b) The argument is entirely similar to that of part a).
\end{proof}
\noindent
We denote Lebesgue measure in $\R^n$ by $\Vol$.  The following result is well known, but we include the proof to show how elementary it is.

\begin{prop}
\label{3.7}
For bounded $\Omega \subset \R^n$ and $r > 0$, let $r \Omega = \{r P \ | \ p \in \Omega\}$, let $\Lambda \subset \R^n$ be a lattice, and let 
\[ L_{\Omega,\Lambda}(r) = \# r\Omega \cap \Lambda \]
be the \textbf{lattice point enumerator}.  If $\Vol(\partial \Omega) = 0$, then
\[ \lim_{r \ra \infty} \frac{L_{\Omega,\Lambda}(r)}{r^n} = \frac{\Vol\Omega}{\Covol \Lambda}. \]
\end{prop}
\begin{proof}
Let $M \in \GL_n(\R)$ be such that $M \Lambda = \Z^n$.  Then for all $r > 0$, 
\[ \# (r \Omega \cap \Lambda) = \# (r M \Omega \cap \Z^n) \] 
and 
\[ \Vol M \Omega = \frac{\Vol\Omega}{|\det M|^{-1}} = \frac{\Vol \Omega}{\Covol \Lambda}, \]
so we may replace $(\Omega,\Lambda)$ by $(M \Omega,\Z^n)$.  Further, by making a change of variable $r \mapsto \frac{r}{R}$ we may assume $M \Omega \subset (-1,1)^n$.  Since $M \Omega$ is bounded and $\Vol(\partial M\Omega) = 0$, the characteristic function $1_{M \Omega}$ is Riemann integrable.  For any $r \in \Z^+$, $\frac{L_{M \Omega,\Z^n}(r)}{r^n}$ is a Riemann sum 
for $1_{M \Omega}$ and the partition of $[-1,1]^n$ into subsquares of 
side length $\frac{1}{r}$.  The result follows.    
\end{proof}

\begin{cor}
\label{COR19}
Let $K$ be an imaginary quadratic field.  Then for all $n \in \Z^+$, 
\begin{equation}
\label{LINCONIMAGQUADEQ}
 C(\Z_K,n) \geq \left( \frac{\pi}{2 \sqrt{|\Delta_K|}} \right)^n. 
\end{equation}
\end{cor}
\begin{proof}
Step 1: The complex place of $K$ gives an embedding $\sigma: K \rightarrow \C$ 
which realizes $\Z_K$ as a lattice in $\C \cong \R^2$; the norm $|\cdot|$ is the square of the usual Euclidean norm.  The lattice $\sigma(\Z_K)$ has covolume\footnote{Later on we will take our Haar measure on $\C$ to be \emph{twice} the standard Lebesgue measure: this would double both the covolume of $\sigma(\Z_K)$ and the volumes of the balls $\{x \ | \ |x| \leq e\}$, so it would not change the final result.} $2^{-1} \sqrt{|\Delta(K)|}$ and 
$\Vol(\{x \in \R^2 \ | \ |x| \leq e\}) = \pi e$. Applying Proposition \ref{3.7} we get that as $e \ra \infty$, 
\[\# \{x \in \Z_K \ | \ |x| \leq e \} \sim \left( \frac{2\pi}{\sqrt{ |\Delta(K)|}} \right) e. \]  
Step 2: We have $A(\Z_K) = \max |1|, |2| = 4$.  For each fixed $\delta > 0$, 
the hypotheses of Theorem \ref{METRICLINTYPETHM}a) apply with $\kappa = \frac{2 \pi}{\sqrt{ |\Delta(K)|}} - \delta$ and thus 
\[ C(\Z_K,n) \geq \left( \frac{\frac{2 \pi}{\sqrt{ |\Delta(K)|}} - \delta}{4} \right)^n. \]
Letting $\delta$ approach zero we get (\ref{LINCONIMAGQUADEQ}).
\end{proof} 

\begin{cor}
For all $n \in \Z^+$, $c(\F_q[t],n) = n-1$.
\end{cor}
\begin{proof} For all $a \in \N$, 
\[\# \{x \in \F_q[t] \ | \ \deg x \leq a \} = q^{a+1}, \]
Applying Theorem \ref{3.5}b) we get $c(\F_q[t],n) \geq n-1$ for all $n \in \Z^+$.  Combining this with Proposition \ref{2.1}b) gives the result.  
\end{proof} 
\noindent
We could now pursue the positive characteristic analogue of Corollary \ref{COR19} by using GoN methods to give bounds on $\{x \in R \ | \ \deg x \leq a \}$.  However, this would involve developing (or importing) GoN methods for Hasse domains of positive characteristic.  But there is a more efficient approach which works for affine domains over an arbitrary ground field: observe that $\{x \in R \ | \ \deg x \leq a\}$ is a Riemann-Roch space and apply (Riemann's portion of) the Riemann-Roch Theorem.  We do so next.



\section{Affine Domains}
\noindent
Let $k$ be a field, $C_{/k}$ a smooth, geometrically integral projective curve, 
and $\infty_1,\ldots,\infty_m$ closed points of $C$ of degrees $d_1,\ldots,d_m$.  Let $C^{\circ} = C \setminus \{\infty_1,\ldots,\infty_m\}$ and $R = k[C^{\circ}]$.  As in $\S 1.8$, we fix $q > 1$ and endow $R$ with the ideal $q$-norm $I \mapsto |I| = q^{\dim_k R/I}$.  We will show that $R$ is of linear $q$-type and give explicit lower bounds on the linear $q$-constants $c(R,n)$.

\subsection{Tornheim's Theorem}
\textbf{} \\ \\ \noindent
It is natural to look first at the case $R = k[t]$, $K = k(t)$.  We have already seen 
that when $k$ is finite, $R$ is of linear type and indeed $c(k[t],n) = n-1$.  
In this section we will show this same result over an arbitrary field $k$.  
In fact a result equivalent to this was first established in a(n apparently little known -- it has no MathSciNet citations as of May 2013) work of L. Tornheim \cite{Tornheim}.  \\ \indent
Tornheim's original proof is more complicated than is necessary.  Our desire to treat a more 
general case also brings certain complications, so we have decided to begin with  a simple proof of Tornheim's Theorem.

\begin{lemma}
\label{TORNLEMMA}
Let $C \in M_n(R) \cap \GL_n(K)$, and let $\Lambda = C R^n$.  Then $\Lambda$ is an $R$-submodule of $R^n$, so we may form the quotient $R$-module $R^n/\Lambda$.  
Then 
\[ \dim_k (R^n/\Lambda) = \deg \det C. \]
\end{lemma}
\begin{proof}
An immediate consequence of Proposition \ref{1.13} and Lemma \ref{1.14}.  
But let us also indicate a direct proof: since $R$ is a PID, 
we may use Smith Normal Form to reduce to the case in which $C$ is diagonal, which is very easy.
\end{proof}

\begin{lemma}
\label{4.2}
\label{LAPP}
(Linear Algebraic Pigeonhole Principle)
Let $V$ be a $k$-vector space and $W_1,W_2$ be linear subspaces 
of $V$.  If $\dim W_1 > \dim V/W_2$, then $W_1 \cap W_2 \neq \{0\}$.  
\end{lemma}
\noindent
The proof is immediate.

\begin{thm}(Tornheim \cite{Tornheim})
\label{TORNTHM}
Let $k$ be a field; let $C = (c_{ij})\in \GL_n(K)$.  For $1 \leq i \leq n$, put $L_i(x) = \sum_{j=1}^n c_{ij} x_j$.  Let $e_1,\ldots,e_n \in \N$ be 
such that 
\begin{equation}
\label{TORNDEGEQ1}
\deg  \det C \leq n-1 + \sum_{i=1}^n e_i.
\end{equation}

Then there exists $x \in (R^N)^{\bullet}$ such that for all $1 \leq i \leq N$, 
\[ \deg L_i(x) \leq e_i. \]
\end{thm}
\begin{proof}
Step 1: We suppose $C \in M_n(R)$.  Consider the linear map $L: K^n \ra K^n, x \mapsto Cx$,
and put $\Lambda = L(R^n) \subset K^N$.  Since $\det C \neq 0$, we have $L^{-1}: K^n \ra K^n$, and thus $L^{-1}|_{\Lambda}: \Lambda \stackrel{\sim}{\ra} R^n$.  Put
\[ \mathcal{B} = \{ (x_1,\ldots,x_n) \in R^n \ | \ \forall 1 \leq i \leq n, \ \deg x_i \leq e_i \}. \]
Then $\mathcal{B}$ is a $k$-subspace of $R^n$ with $\dim_k \mathcal{B} = \sum_{i=1}^n (e_i + 1)$. 
By Lemma \ref{TORNLEMMA}, $\dim_k R^n/\Lambda = \deg \det C$.  So  (\ref{TORNDEGEQ1}) can be restated as
\[ \dim_k \mathcal{B} > \dim_k R^n/\Lambda.\]
By Remark \ref{4.2} there is a nonzero vector $y \in \Lambda \cap \mathcal{B}$.  Taking $x = L^{-1} y$ does the job.  \\
Step 2: In the general case, choose $f \in R^{\bullet}$ such that $f C \in M_n(R)$.  Then \[\deg \det fC \leq n-1 +  \sum_{i=1}^n (e_i + \deg f), \]
so by Step 1, there is $x \in (R^n)^{\bullet}$ with $\deg f + \deg L_i(x) = \deg f L_i(x) \leq e_i + \deg f$ for all $1 \leq i \leq n$, so $\deg L_i(x) \leq e_i$ 
for all $1 \leq i \leq n$.
\end{proof}

\begin{cor}
\label{TORNHEIMTHM}
\label{4.4}
For all $n \geq 1$, $c(k[t],n) = n-1$.
\end{cor}
\begin{proof}
Since $k[t]$ is a PID, all $R$-lattices are free, so by Remark \ref{2.6} 
the special case of the linear $q$-type condition we've checked is equivalent 
to the general case: $c(k[t],n) \geq n-1$.  The upper bound comes from Proposition \ref{2.1}. 
\end{proof}

\subsection{Affine Domains Are Of Multimetric Linear Type}

\begin{thm}
\label{4.5}
 Let $k$ be a field and $C^{\circ}/k$ be a nice affine curve of genus $g$.  Let $R = k[C^{\circ}]$ be its affine 
coordinate ring.  Let $d = \min \deg \infty_1,\ldots, \infty_m$.  Then
\[ c(R,n) \geq n(2-d-g) -1. \]
\end{thm}
\begin{proof} 
We will show that if $m = 1$ and $\deg \infty_1 = d$, then $c(R,n) \geq n(2-d-g) -1$.  By Proposition \ref{2.7}, this suffices.  Let $\Lambda \subset R^n$ be an integral $R$-lattice such that 
\[ \covol \Lambda \leq n(2-d-g)-1 + \sum_{i=1}^n e_i. \]
We must show there is $x = (x_1,\ldots,x_n) \in \Lambda^{\bullet}$ with 
$\deg x_i \leq e_i$ for all $1 \leq i \leq n$.  By Lemma \ref{1.14}, 
$\covol \Lambda = \dim_k R^n/\Lambda$.   Consider the $k$-vector space
\[ \mathcal{B} = \{(x_1,\ldots,x_n) \in R^n \ | \ \deg x_i \leq e_i\}. \]
By Proposition \ref{1.12}, for $x \in K$, $\deg x = -d \ord_{\infty}(x)$, so for $x \in R$, 
\[ \deg x \leq e_i \iff -d \ord_{\infty}(x) \leq e_i \iff 
\ord_{\infty}(x) \geq \lceil - \frac{e_i}{d} \rceil = - \lfloor \frac{e_i}{d} \rfloor. \]
Thus $\{ x \in R \ | \ \deg x \leq e_i\}$ is precisely the 
Riemann-Roch space $\mathcal{L}( (\lfloor \frac{e_i}{d} \rfloor) \infty)$, so 
by Riemann-Roch its dimension is at least $d \lfloor \frac{e_i}{d} \rfloor - g + 1$.  Therefore 
\[ \dim+k \mathcal{B} \geq \sum_{i=1}^n \left(d \lfloor \frac{e_i}{d} \rfloor - g + 1 \right) \geq d \sum_{i=1}^n\left( \frac{e_i}{d} - \frac{d-1}{d} \right) - ng + n = 
n(2-d-g) + \sum_{i=1}^n e_i. \]
It follows that
\[\dim_k \mathcal{B} > n(2-d-g)-1 + \sum_{i=1}^n e_i \geq \covol \Lambda = 
\dim_k R^n/\Lambda, \]
so by the Linear Algebraic Pigeonhole Principle $\mathcal{B} \cap \Lambda^{\bullet} \neq \varnothing$.  
\end{proof}

\begin{thm}
\label{4.7}
 Let $k$ be a field and $C/k$ be a nice projective curve of genus $g$, and let $C^{\circ} = C \setminus \{\infty_1,\ldots,\infty_m\}$ be the affine curve obtained by removing the given $m$ closed points, of degrees $d_1,\ldots,d_m$.   Let $R = k[C^{\circ}]$ be its affine 
coordinate ring.   Then $R$ is of multimetric linear type, and 
\[ c_M(R,n) \geq n(m+1 - \sum_{j=1}^m d_j - g) - 1.\]
\end{thm}
\begin{proof}
Let $\Lambda \subset R^n$ be an integral $R$-lattice such that 
\[ \covol \Lambda \leq n(m+1-\sum_{j=1}^md_j-g)-1 + \sum_{i,j} e_{ij}. \]
We must show there is $x = (x_1,\ldots,x_n) \in \Lambda^{\bullet}$ with 
$\deg_j x_i \leq e_{ij}$ for all $i,j$.  By Lemma \ref{1.14}, 
$\covol \Lambda = \dim_k R^n/\Lambda$.  Consider
the $k$-vector spaces 
\[\forall 1 \leq i \leq n, \ \mathcal{B}_i = \{x_i \in R \ | \ \forall 1 \leq j \leq m, \ \deg_j x_i \leq e_{ij}\}, \]
\[ \mathcal{B} = \{(x_1,\ldots,x_n) \in R^n \ | \ \deg_j x_i \leq e_{ij}\} = \prod_{i=1}^n \mathcal{B}_i.\]
By Proposition \ref{1.12}, for $x_i \in K$, \[\deg_j x_i \leq e_{ij} \iff 
-d_j \ord_{\infty_j} x_i \leq e_{ij} \iff \ord_{\infty_j} x_j \geq \frac{-e_{ij}}{d_j}. \]
Thus $\mathcal{B}_i$ is the Riemann-Roch space $\mathcal{L}(\sum_{j=1}^m 
\lfloor \frac{e_{ij}}{d_j} \rfloor \infty_j)$.  By Riemann-Roch, 
\[ \dim_k \mathcal{B}_i \geq \sum_{j=1}^m d_j \lfloor \frac{e_{ij}}{d_j} \rfloor - g+1 \]
\[ \geq \sum_{j=1}^m d_j\left(\frac{e_{ij}}{d_j} - \frac{d_{j-1}}{d_j} \right) - 
g+ 1 = 
\sum_{j=1}^m e_{ij} - \sum_{j=1}^m d_j + m - g + 1. \]
It follows that
\[\dim_k \mathcal{B} > \sum_{i,j} e_{ij} + n(m+1-\sum_{j=1}^m d_j - g) - 1 \geq 
\covol \Lambda, \]
so by the Linear Algebraic Pigeonhole Principle $\mathcal{B} \cap \Lambda^{\bullet} \neq \varnothing$.  
\end{proof}

\section{Blichfeldt, Minkowski and Hasse Domains}

\subsection{Abstract Blichfeldt and Minkowski Theorems}

\begin{prop}(Measure Theoretic Pigeonhole Principle)
Let $(X,\mu)$ be a measure space, $\{S_i\}_{i \in I}$ a countable family of measurable subsets of $X$, $m \in \N$.  If 
\begin{equation}
\label{MPHP}
\sum_{i \in I} \mu(S_i) > m \mu(\bigcup_{i \in I} S_i), 
\end{equation}
then there is $x \in X$ with $\# \{i \in I \ | \ x \in S_i\} > m$.
\end{prop}
\begin{proof}
By replacing $X$ with $\bigcup_{i \in I} S_i$ we may assume that 
$\bigcup_{i \in I} S_i = X$.  Further, it is no loss of generality 
to assume that $\mu(X) > 0$ and that no $x \in X$ lies in infinitely 
many of the sets $S_i$: indeed, in the former case the hypothesis does 
not hold and in the latter case the conclusion holds.  \\ \indent
For a subset $S \subset X$, denote by $1_S$ the associated characteristic function: $1_S(x) = 1$ if $x \in S$, and otherwise $1_S(x) = 0$.  Put 
\[ f = \sum_{i \in I} 1_{S_i}.\]
For any $x \in X$, $f(x) = \# \{i \in I \ | \ x \in S_i\}$, so $f: X \ra \R$ is a measurable function.  The condition (\ref{MPHP}) can be reexpressed as \[\int_X f d\mu > m \int_X d \mu, \]
so we must have $\# \{i \in I \ | \ x \in S_i\}= f(x) > m$ for at least one $x \in X$. 
\end{proof}
\noindent
A \textbf{measured group} $(G,+,\mathcal{A},\mu)$ is a group $(G,+$) -- not assumed to be commutative, though we write the group law additively -- and a 
measure $(G,\mathcal{A},\mu)$ which is right invariant: for all $A \in \mathcal{A}$ and $x \in G$, $\mu(A+x) = \mu(A)$.  To avoid trivialities, we 
assume $\mu(G) > 0$.  
\\ \\
Let $\Gamma$ be a subgroup of $G$.  A \textbf{fundamental domain} $\mathcal{F}$ 
for $\Gamma$ in $G$ is a measurable subset $\mathcal{F} \subset G$ such that \\
(FD1) $\bigcup_{g \in \Gamma} \mathcal{F} + g = \Gamma$, and \\
(FD2) Fo all $g_1,g_2 \in \Gamma$, $\mu( (\mathcal{F}+g_1) \cap (\mathcal{F}+g_2)) = 0$.  

\begin{lemma}
If $\mathcal{F}_1$ and $\mathcal{F}_2$ are both fundamental domains for a countable subgroup $\Gamma$ in $G$, then $\mu(\mathcal{F}_1) = \mu(\mathcal{F}_2)$.
\end{lemma}
\begin{proof}
Observe that if $\{\mathcal{S}_i\}_{i \in I}$ is a countable family 
of subsets such that $\mu(S_i \cap S_j) = 0$ for all $i \neq j$, then 
\[\mu (\bigcup_{i \in I} S_i) = \sum_{i \in I} \mu(S_i).\]
Now we have
\[ \mathcal{F}_1 \supset \mathcal{F}_1 \cap (\bigcup_{g \in \Gamma} \mathcal{F}_2 + g) = \bigcup_{g \in \Gamma} \mu( \mathcal{F}_1 \cap (\mathcal{F}_2)+g), \]
so, using the above observation, 
\[ \mu(\mathcal{F}_1) \geq \sum_{h \in H} \mu(\mathcal{F}_1 \cap (\mathcal{F}_2 + g)) = \sum_{g \in \Gamma} \mu(\mathcal{F}_1 \cap (\mathcal{F}_2 - g)) = 
\sum_{g \in \Gamma} \mu((\mathcal{F}_1+g) \cap \mathcal{F}_2) \]
\[ = \mu( \bigcup_{g \in \Gamma} (\mathcal{F}_1 + g) \cap \mathcal{F}_2) = \mu(\mathcal{F}_2). \]
Interchanging $\mathcal{F}_1$ and $\mathcal{F}_2$ we get the result.
\end{proof}
\noindent
A subgroup $\Lambda$ of a measured group $G$ is a \textbf{lattice} 
if it is countable and admits a measurable fundamental domain of finite measure.  
We define the \textbf{covolume} $\Covol \Lambda$ to be the measure of any such 
fundamental domain.  Note that our assumption $\mu(G) > 0$ implies $\Covol \Lambda > 0$.  

\begin{thm}(Abstract Blichfeldt Lemma)
\label{ABSTRACTBLICHFELDT}
Let $\Lambda$ be a lattice in a measured group $G$, and let $M \in \Z^+$.  
Let $\Omega \subset G$ be measurable, and suppose 
\begin{equation}
\frac{ \mu(\Omega)}{\Covol \Lambda} > M.
\end{equation}
There are distinct $w_1,\ldots,w_{M+1} \in \Omega$ such that for all 
$1 \leq i,j \leq M+1$, $w_i - w_j \in \Lambda$.
\end{thm}
\begin{proof}
Let $\mathcal{F}$ be a measurable fundamental domain for $\Lambda$ in $G$.  
For $x \in \Lambda$, let 
\[ \Omega_x = \Omega \cap (\mathcal{F} + x). \]
Then $\Omega = \bigcup_{x \in \Gamma} \Omega_x$: this is a countable union 
which is essentially pairwise disjoint -- for all $x \neq y \in \Gamma$, 
$\mu(\Omega_x \cap \Omega_y) = 0$ -- so  
\begin{equation}
\label{ABBLICHEQ}
\sum_{x \in \Gamma} \mu(\Omega_x - x) = \sum_{x \in \Lambda} \mu(\Omega_x) = 
\mu(\Omega) > M \Covol(\Lambda) = M \mu(\mathcal{F}). 
\end{equation}
We apply the Measure Theoretic Pigeonhole Principle with $X = \mathcal{F}$, 
$I = \Lambda$, $S_x = \Omega_x - x$: there is $v \in \mathcal{F}$ and 
$x_1,\ldots,x_{M+1} \in \Lambda$ such that \[v \in \bigcap_{i=1}^{M+1} \Omega_{x_i} - x_i. \]
Thus for $1 \leq i \leq  M+1$ there is $w_i \in \Omega_{x_i}$ -- so $w_1,\ldots,w_{M+1}$ are distinct -- with
\[ \forall 1 \leq i \leq M+1, w_i - x_i = v. \]
It follows that for all $1 \leq i,j \leq M+1$, $w_i - w_j = (x_i+v) - (x_j+v) = 
x_i - x_j \in \Lambda$.  
\end{proof}

\begin{remark}
When $\mu$ is the counting measure on $G$, we essentially recover the Group Theoretic Pigeonhole Principle: precisely, the special case in which $\Gamma$ is countable and $\kappa$ is finite.
\end{remark}
\noindent
A \textbf{measured ring} is a ring endowed with a measure such that the 
additive group of $R$ is a measured group.  Again we assume $\mu(R) > 0$ 
to avoid trivialities.  

\begin{thm}(Abstract Minkowski Theorem)
\label{AMCBT}
\label{5.5}
Let $M\in \Z^+$, $(R,+,\cdot,\mathcal{A},\mu)$ be a measured ring, and let $\Lambda \subset R^N$ be a countable subgroup.  Let $\Omega \subset R$ be 
measurable and symmetric: $x \in \Omega \implies -x \in \Omega$.  \\
a) We suppose $2 \in R^{\bullet}$ and all of the following: \\
$\bullet$ $\Omega$ is \textbf{midpoint closed}: $x,y \in \Omega \implies \frac{x+y}{2} \in \Omega$.  \\
$\bullet$ $2 \Lambda$ is a lattice in $R$.  \\
$\bullet$ $\frac{\mu(\Omega)}{\Covol 2\Lambda} > M$.  \\
Then $\# (\Omega \cap \Lambda^{\bullet}) \geq M$.  \\
b) We suppose all of the following: \\
$\bullet$ $\Omega$ is closed under subtraction: $x,y \in \Omega \implies x-y \in \Omega$.  \\
$\bullet$ $\Lambda$ is a lattice in $R$. \\
$\bullet$ $\frac{\mu(\Omega)}{(\Covol\Lambda} > M$.  \\
Then  $\# (\Omega \cap \Lambda^{\bullet}) \geq M$. 
\end{thm}
\begin{proof}
a) Apply the Abstract Blichfeldt Lemma with $G = (R,+)$ and $2\Lambda$ in place 
of $\Lambda$.  We get distinct elements $w_1,\ldots,w_{M+1} \in \Omega$ such that for all $1 \leq i,j \leq M+1$, $\frac{w_i - w_j}{2} \in \Lambda$.  Since 
$\Omega$ is symmetric and midpoint closed, $-w_j \in \Omega$ and thus $\frac{w_i-w_j}{2} \in \Omega$ for all $1 \leq i,j \leq M+1$.  Fixing $i = 1$ 
and letting $j$ run from $2$ to $M+1$ gives us $M$ nonzero elements of $\Omega \cap \Lambda$.  \\
b) This is exactly the same as part a) except we use $\Lambda$ instead of $2\Lambda$ and use the fact that $\Omega$ is closed under subtraction.
\end{proof}

\begin{cor}(Minkowski Convex Body Theorem)
Let $\Omega \subset \R^n$ be symmetric and convex, and let $\Lambda \subset \R^n$ be a lattice.  If $\Vol \Omega > 2^n \Covol \Lambda$, then $\Omega \cap \Lambda^{\bullet} \neq \varnothing$.
\end{cor}
\begin{proof}
A convex subset is midpoint closed.  Also $\Covol (2\Lambda) = 2^n \Covol \Lambda$.  Now apply Theorem \ref{5.5}a).
\end{proof}

\begin{cor}(Chonoles Convex Body Theorem \cite{Chonoles})
Let $\mathcal{R} =  \F_q((\frac{1}{t}))^n$ and let $\Lambda$ be an $\F_q[t]$-lattice in $R$.  If $\Omega \subset \mathcal{R}$ is closed under subtraction and satisfies $\Vol \Omega > \Covol \Lambda$, then $\Omega \cap \Lambda^{\bullet} \neq 0$.
\end{cor}
\noindent
Let $R = \Z_{K,S}$ be an $S$-integer ring in a number field $K$.  Let $\mathcal{R} = \prod_{v \mid \infty \cup v \in S} K_v$.  This is a locally 
compact ring.  We endow it with the product of the Haar measures on each factor, where each factor isomorphic to $\R$ gets the standard Lebesgue measure, each factor isomorphic to $\C$ gets twice the standard Lebesgue measure, and each 
non-Archimedean local field $K_v$ gets the Haar measure which gives its maximal compact subring $\mathcal{O}_v$ volume $1$.  It is a standard fact that $\sigma(R)$ is discrete and cocompact in $\mathcal{R}$: see e.g. \cite{Conrad}.  
Let $\mathcal{V}(R)$ denote the $\mu$-volume of a fundamental domain for $\sigma(R)$ in $\mathcal{R}$.  \\ \indent
On $\mathcal{R}^n$, let $\mu$ the product Haar measure.  Let $\Lambda \subset K^n$ be an $R$-sublattice, and let $\hat{\sigma}: K^n \ra \mathcal{R}^n$ be the natural embedding.  It follows that $\hat{\sigma}(\Lambda)$ is discrete and cocompact in $\mathcal{R}^n$, and that its covolume in the measure theoretic sense is equal to $|\chi(\Lambda)| \mathcal{V}(R)^n$.  Thus if we take $\Vol$ 
to be $\mathcal{V}(R)^{-n} \mu$, then $\Vol$ is a Haar measure on $\mathcal{R}^n$ 
such that $\Covol \Lambda$ means both the covolume in the sense of $\S$ 1.9 and 
the measure of a fundamental domain for $\Lambda$ in $\mathcal{R}^n$.


\begin{cor}
\label{5.8}
Let $\Omega \subset \mathcal{R}^n$ be a measurable subset such that 
\[ \Vol \Omega > \Covol \Lambda. \]
Then there are distinct $x,y \in \Omega$ such that $x-y \in \Lambda$.
\end{cor}

\subsection{Hasse Domains Are of Multimetric Linear Type}
\textbf{} \\ \\ \noindent
For $z = x+yi \in \C$, recall that we have taken the normalization $|z| = x^2+y^2$.

\begin{lemma}
\label{ALGNT2}
a) When $S = \varnothing$, $\mathcal{V}(R) = |\Delta(K)|^{\frac{1}{2}}$. \\
b) Let $r,s \in \N$, $d = r+2s$, $t \in \R$, and let 
\[ B_t = \{ (y_1,\ldots,y_r,z_1,\ldots,z_s) \in \R^r \times \C^s \ | \ 
\sum_{i=1}^r |y_i| + 2 \sum_{j=1}^s |z_j|^{\frac{1}{2}} \leq t \}. \]
Then for all $t \geq 0$, \[\Vol B_t = \mathcal{V}(R)^{-n} 2^r \pi^s 
\frac{t^d}{d!}. \]
\end{lemma}
\begin{proof}
Both assertions are part of the standard application of geometry of numbers to algebraic number theory.  For proofs see e.g. \cite[Ch. IV]{Samuel}.
\end{proof} 


\begin{thm}
\label{NBR1}
\label{5.11}
Let $K$ be a number field.  Suppose $K$ has $r$ real places and $s$ complex 
places, and put $d = r+2s = [K:\Q]$.  Let $\Z_K$ be the ring of integers of $K$, and let $\Z_K \subset R \subset K$ be an overring of $R$.  Then for all 
$n \in \Z^+$, 
\begin{equation}
\label{MINKLINEQ}
 C(R,n) \geq M(K)^{-n}, 
\end{equation}
where 
\[ M(K) = \left( \frac{4}{\pi} \right)^s \frac{d!}{d^d} \left| \Delta_K \right|^{\frac{1}{2}} \]
is Minkowski's constant.
\end{thm}
\begin{proof}
By Proposition \ref{2.6}, we may assume $R = \Z_K$.  Consider the embedding $\hat{\sigma}: K^n \hookrightarrow \R^{nd}$, $(x_1,\ldots,x_n) \mapsto (\sigma(x_1),\ldots,\sigma(x_n))$.  For any $R$-sublattice $\Lambda$ of $R^n$, 
$\hat{\sigma}(\Lambda)$ is a lattice in $\R^{nd}$.  For $\epsilon = (\epsilon_1,\ldots,\epsilon_n) \in (\R^{>0})^n$, we define $S_1(\epsilon) \subset \R^{nd}$ as the set of all $x \in \R^{nd}$ satisfying 
\[ |x_{i1}| \cdots |x_{ir}| |x_{i(r+1)}^2 + x_{i(r+2)}^2| \cdots 
|x_{i(d-1)}^2 + x_{i d}^2| \leq \epsilon_i \]
for $1 \leq i \leq n$.  By the arithmetic geometric mean inequality, $S_1(\epsilon)$ contains the 
symmetric compact convex body $S_2(\epsilon)$ defined by 
\[ |x_{i1}| + \ldots + |x_{ir}| + 2 | x_{i(r+1)}^2 + x_{i(r+2)}^2|^{\frac{1}{2}} + \ldots + 2|x_{i(d-1)}^2 + x_{id}^2|^{\frac{1}{2}} \leq d \epsilon_i^{\frac{1}{d}} \]
for all $1 \leq i \leq n$.  By Lemma \ref{ALGNT2}, 
\[\Vol S_2(\epsilon) = \mathcal{V}(R)^{-n} \left(2^r \pi^s \frac{d^d}{d!} \right)^n \epsilon_1 \cdots \epsilon_n. \]
Applying Minkowski's Convex Body Theorem, there is a nonzero point of $\Lambda$ 
in $S_2(\epsilon)$ (hence also in $S_1(\epsilon)$ if 
\[ \mathcal{V}(R)^{-n} \left(2^r \pi^s \frac{d^d}{d!} \right)^n (\epsilon_1 \cdots \epsilon_N) = \Vol(S_2(\epsilon)) \geq 2^{nd}, \]
i.e., if and only if
\[ \Covol \Lambda \leq M(K)^{-n} \epsilon_1 \cdots \epsilon_n. \]
Since a point in $S_1(\lambda)$ satisfies $|x_i| \leq \epsilon_i$ for all $i$, 
this shows that $M(K)^{-n}$ is a linear constant for $\Z_K$ in dimension $n$.  
\end{proof}

\begin{remark}
When $K$ is an imaginary quadratic field, the lower bound on $C(\Z_K,n)$ 
given in (\ref{MINKLINEQ}) is precisely the same lower bound given in 
(\ref{LINCONIMAGQUADEQ}).  
\end{remark} 

\begin{thm}
\label{5.10}
\label{NBR2}
The Hasse domain $R = \Z_{K,S}$ is of multimetric linear type.  More precisely, 
for all $n \in \Z^+$, 
\[ C_M(R,n) \geq (\frac{\pi}{4})^s \mathcal{V}(R)^{-n}. \]
\end{thm}
\begin{proof}
Consider the embedding $\hat{\sigma}: K^n \hookrightarrow \mathcal{R}^n$, $(x_1,\ldots,x_n) \mapsto (\sigma(x_1),\ldots,\sigma(x_n))$.  For any $R$-sublattice $\Lambda$ of $R^n$, 
$\hat{\sigma}(\Lambda)$ is a lattice in $\R^{nd}$ with covolume $[R^n:\Lambda]$.  Fix $\epsilon_{i,j} > 0$ and $e_{i,P_j} \in \Z$.  Let 
$\Omega$ be the set of $(x_1,\ldots,x_n) \in \mathcal{R}^n$ such that 
$|x_i|_{\infty_j} \leq \epsilon_{ij}$ and $\deg_{P_j} x_i 
\leq e_{i P_j}$ for all $i$ and $j$.  Then 
\[ \Vol \Omega = \mathcal{V}(R)^{-n} 2^r \pi^s \prod_{i,j} \epsilon_{ij} \prod_{j=1}^s q_{P_j}^{e_{1 P_j} + \ldots + e_{n P_j}}. \]
By Corollary \ref{5.8}, if 
\[ 2^r \pi^s \prod_{i,j} \epsilon_{ij} \prod_{j=1}^s q_{P_j}^{e_{1 P_j} + \ldots + e_{n P_j}} > [R^n:\Lambda] \mathcal{V}(R)^n, \]
then there are distinct $x,y \in \Omega$ such that $x-y \in \Lambda$.  Thus if 
we define $\epsilon_{ij}'$ to be $\frac{\epsilon_{ij}}{2}$ when $\infty_j$ is 
real and $\frac{\epsilon_{ij}}{4}$ when $\infty_j$ is complex, we find that 
if 
\[ 2^r \pi^s \prod_{i,j} \epsilon_{ij}' \prod_{j=1}^s q_{P_j}^{e_{1 P_j} + \ldots + e_{n P_j}} > [R^n:\Lambda] \mathcal{V}(R)^n, \]
then there is $x \in \Omega \cap \Lambda^{\bullet}$.  Thus any number larger 
than $2^r \pi^s 2^{-r-2s} \mathcal{V}(R)^{-n} = 2^{-2s} \pi^s \mathcal{V}(R)^{-n}$ is a linear multimetric constant in dimension $n$.  
\end{proof}


\begin{remark}
In Theorem \ref{5.10} we are more concerned with the qualitative result that 
$S$-integer rings in number fields are of multimetric linear type than with the 
value of the constant.  Thus we have not employed the arithmetic geometric mean maneuver of the proof of Theorem \ref{5.8} or even computed $\mathcal{V}(R)$ in the general case.
\end{remark}

\section{Quadratic Forms: the Nullstellensatz}

\subsection{The Nullstellensatz}
\textbf{} \\ \\ \noindent
Let $(R,|\cdot|)$ be a normed Dedekind domain of multimetric linear type.  We renormalize so that each norm $|\cdot|_j$ has Artin constant at most $2$.  Thus the triangle inequality holds and $|n|_j \leq n$ for all Archimedean $j$ and 
$|n|_j = 1$ for all non-Archimedean $j$.  
\\ \\
For $x = (x_1,\ldots,x_n) \in K^n$, we put
\[ |x|_j = \max_i |x_i|_j, \ |x| = \prod_{j=1}^m |x|_j. \]
For a matrix $M = (m_{ij}) \in M_n(R)$, we put 
\[ |M|_j = \max_{i,k} |m_{ik}|_j \]
when $j$ is a $q$-norm and 
\[ |M|_j = \sum_{i,k} |m_{ik}|_j \]
when $j$ is Archimedean.  Also put 
\[ |M| = \prod_{j=1}^m |M|_j. \]
For a quadratic form 
\[ f = \sum_{1 \leq i \leq k \leq n} m_{ik} x_i x_j \]
with coefficients in $R$, we let $M = (m_{ik})$ be the corresponding upper 
triangular matrix and put 
\[ |f| = |M|. \]
 An \textbf{isotropic vector} for a quadratic form $f$ is a nonzero vector $v \in R^n$ with $f(v) = 0$.  A form $f$ is \textbf{isotropic} if it has an isotropic vector and otherwise \textbf{anisotropic}.
\\ \\
There is a rival notion of the size of a vector $x = (x_1,\ldots,x_n) \in K^n$, namely we could take $\max_i |x_i|$.  Curiously, this is the measure we will 
use in the statement of the Nullstellensatz, although both will occur in the proof!  For later use we note the relationship between them:

\begin{equation}
\label{NORMNORMEQ}
\forall x = (x_1,\ldots,x_n) \in K^n, \ \max_i |x_i| \leq |x|.
\end{equation}

\begin{thm}[Nullstellensatz]
\label{6.1}
Let $(R,|\cdot|)$ be a normed Dedekind domain of multimetric linear type, with 
fraction field $K$.  We suppose that $|R|$ is a discrete subset of $\R$ (e.g. if some equivalent norm is $\Z$-valued).  Let $f = \sum_{i,j} m_{ij} t_i t_j \in R[t_1,\ldots,t_n]$ be a nonzero isotropic quadratic form. \\
a) If the norm is of $q$-type, then $f$ admits an isotropic vector $a \in R^n$ with 
\begin{equation}
\label{ISOTROPYEQ1}
\max_i \deg a_i \leq mn-1 - c_M(R,n) + \left( \frac{n-1}{2} \right) \deg f. 
\end{equation}
b) Suppose there is at least one $j$ such that $|\cdot|_j$ is Archimedean 
and $m'$ of them altogether.  Let $0 < \delta < C_M(R,n)$ and $0 < \eta < 1$.  Then $f$ admits 
an isotropic vector $a$ with 
\begin{equation}
\label{GRANDEQ}
\max_i |a_i| \leq \max  \left( \frac{ (1-\eta^{\frac{1}{m'}})^{-m'(n-1)}}{\eta (C_M(R,n)-\delta)} \right),\left(\frac{1}{\eta (C_M(R,n)-\delta)} \right) |3f|^{\frac{n-1}{2}} . 
\end{equation}
c) If $R = \Z$ then $q$ admits an isotropic vector $a \in R^n$ with 
\[ |a| \leq (3|f|)^{\frac{n-1}{2}}. \]

\end{thm}
\begin{proof} In Steps 0 and 1 we treat the part of the proof which is essentially the same in all cases.  Then we treat the $q$-normed case in Step 2, 
the densely normed case in Step 3, and the $R = \Z$ case in Step 4. \\ 
Step 0:  Since $|R|$ is discrete, there is an $f$-isotropic vector $a = (a_1,\ldots,a_n) \in R^n$ with $|a|$ minimal.   By permuting the variables, we may assume that $\max_i |a_i| = |a_n|$.  For $x,y \in K^n$, we define a bilinear form
\[ \langle x, y \rangle = f(x+y)- f(x)- f(y) = \sum_{1 \leq i \leq j \leq n} m_{ij}(x_i y_j + x_j y_i). \]
Step 1: We {\sc claim}: for all $b \in R^n$ with $f(b) \neq 0$ and all $c \in K$, 
\begin{equation}
\label{ISOEQ3}
|f|^{-1} \leq  |3| |b-ca|^2. 
\end{equation}
{\sc proof of claim:}
Let 
\[a^* = f(b)a - \langle a,b \rangle b. \]
A calculation -- which can be interpreted in terms of reflection through $b$ -- gives  
\[ f(a^*) = f(b)^2f(a) - f(b) \langle a,b \rangle 
\langle a,b \rangle   + \langle a,b \rangle^2 f(b)= 0. \]
By the minimality of $a$, we have 
\begin{equation}
\label{ISOEQ1}
|a| \leq |a^*|.
\end{equation}
Now put $d = b-ca$, so $b = d+ca$.  Then 
\[ a^* = f(d+ca)a - \langle a,d+ca \rangle (d+ca) = f(d) a - \langle a,d \rangle d. \]
Fix $1 \leq j \leq m$.  Suppose first that $|\cdot|_j$ is Archimedean.  Then: 
\[ |a^*|_j = |f(d)a - \langle a,d \rangle d|_j \leq 
|\sum_{i,j} m_{ij} d_i d_j|_j |a|_j + |2|_j |\sum_{i,j} m_{ij} a_i d_j|_j |d|_j \]
\[ \leq \sum_{i,j} |m_{ij}|_j |a|_j |d|_j^2 + |2|_j \sum_{i,j} |m_{ij}|_j |a|_j |d|_j^2 
=  |3|_j|f|_j |a|_j |d|_j^2. \]
The ultrametric case is similar.  Multiplying from $j = 1$ to $m$ we get  
\begin{equation}
\label{ISOEQ2}
|a^*| \leq  |3| |f| |a| |d|^2. 
\end{equation}  
Combining (\ref{ISOEQ1}) and (\ref{ISOEQ2}) and dividing through 
by $|a| |f|$, we get (\ref{ISOEQ3}).  \\
Step 2: Suppose the norm is of $q$-type.  We may assume that 
\[\deg a_n \geq mn-c_M(R,n), \] for (\ref{ISOTROPYEQ1}) holds otherwise. 
  Apply Theorem \ref{2.10}b) with $n-1$ in place of $n$, $M = \deg a_n$ and $\theta_i = \frac{a_i}{a_n}$ for $1 \leq i \leq n-1$: there is $b = (b_1,\ldots,b_n) \in R^n$ with 
$b_n \neq 0$ such that 
\begin{equation}
\label{STEP4EQ}
\forall 1 \leq i \leq n-1, \forall 1 \leq j \leq m, \ \deg_j b_n \theta_i - b_i \leq \frac{mn-1-c_M(R,n) - \deg a_n}{m(n-1)}, 
\end{equation}
\begin{equation}
\label{STEP4EQ2}
\forall 1 \leq j \leq m, \  \deg_j b_n \leq \deg_j a_n - 1.
\end{equation}
We {\sc claim} that for all $i,j$, $\deg_j b_i \leq \deg_j a - 1$.  When $i = n$, this follows from (\ref{STEP4EQ2}).  Suppose $1 \leq i \leq n-1$.  Then for all $1 \leq j \leq m$,  
 \[ \deg_j b_i \leq \max \left(\frac{mn-1-c_M(R,n) - \deg a_n}{m(n-1)}, \deg_j b_n + 
\deg_j a_i - \deg_j a_n \right).\]
If $b_i = 0$ then 
for all $j$, $\deg_j b_i = -\infty \leq \deg_j a_n - 1$. If $b_i \neq 0$, then since $\left(\frac{mn-1-c_M(R,n) - \deg a_n}{m(n-1)}\right) < 0$, we have
\[\deg_j b_i \leq \deg_j a_i + (\deg_j b_n - \deg_j a_n ) \leq \deg_j a_i -1, \]
establishing the claim.  Thus for all $j$, $\deg_j b = \max_i \deg_j b_i \leq 
\deg_j a - 1$, so 
\[ \deg b = \sum_{j=1}^m \deg_j b \leq \deg_j a - m < \deg a, \]
so by minimality of $a$, $f(b) \neq 0$.  
Put $c = \frac{b_n}{a_n}$; then 
\begin{equation}
\label{ISOEQ4}
 \deg b-ca \leq \frac{mn-1-c_M(R,n) - \deg a_n}{n-1}.
\end{equation}
In this case (\ref{ISOEQ3}) can be restated as 
\begin{equation}
\label{ISOEQ3BIS}
- \deg f \leq 2\deg(b-ca). 
\end{equation}
Combining (\ref{ISOEQ3BIS}) and (\ref{ISOEQ4}) we get 
\[ - \deg f \leq 2 \left( \frac{mn-1-c_M(R,n) - \deg a_n}{n-1} \right), \]
which is equivalent to 
\[ \max_i \deg a_i = \deg a_n \leq (mn-1) - c_M(R,n) + \left(\frac{n-1}{2}\right) \deg f. \]\
Step 3: Suppose that the number $m'$ of Archimedean infinite places is at least one.  For $\delta > 0$, we put $C_{\delta} = C_M(R,n) - \delta$: we will take $\delta$ to be small enough so that $C_{\delta} > 0$.  We introduce an auxiliary 
parameter $\eta \in (0,1)$.  From the form of the claimed inequality on 
$|a_n| = \max_i |a_i|$ we may assume 
\begin{equation}
\label{NASTY1}
|a_n| > \frac{(1-\eta^{\frac{1}{m'}})^{-m'(n-1)}}{\eta C_{\delta}}. 
\end{equation}
Let us also put 
\[ \kappa = \left( C_{\delta} \eta |a_n| \right)^{\frac{-1}{m'(n-1)}};\]
then (\ref{NASTY1}) is equivalent to 
\begin{equation}
\label{NASTY2} 
\kappa < 1 - \eta^{\frac{1}{m'}}.
\end{equation}
Apply Theorem \ref{2.10} with $n-1$ in place of $n$, $M = \eta |a_n|$ and $\theta_i = \frac{a_i}{a_n}$ for $1 \leq i \leq n-1$: there is $b = (b_1,\ldots,b_n) \in R^n$ with 
$b_n \neq 0$ such that: if $j$ is Archimedean then 
\[ \forall 1 \leq i \leq n-1, \  |b_n \theta_i - b_i|_j \leq \left( C_{\delta} \eta |a_n|\right)^{\frac{-1}{m'(n-1)}} = \kappa < 1\]
\[ |b_n|_j \leq \eta^{\frac{1}{m'}} |a_n|_j, \]
whereas for every non-Archimedean $j$ we have 
\[ \forall 1 \leq i \leq n-1, \ |b_n \theta_i - b_i|_j \leq 1, \]
\[ |b_n|_j \leq 1. \]
For all non-Archimedean $j$ and all $i$ we have $|b_i|_j \leq 1 \leq |a_n|_j$ and thus $|b|_j \leq |a|_j$.  Now let $j$ be Archimedean.  We have 
\[|b_n|_j \leq \eta^{\frac{1}{m'}} |a_n|_j < |a_n|_j. \]
Further, for $1 \leq i \leq n-1$,  
\[ |b_i|_j = |b_n \frac{a_i}{a_n} - b_i - b_n \frac{a_i}{a_n}|_j 
\leq |b_i - b_n \frac{a_i}{a_n}|_j + \frac{|b_n|_j}{|a_n|_j} |a_i|_j \leq \kappa + \eta^{\frac{1}{m'}} |a_i|_j . \]
If $a_i = 0$, then this gives \[|b_i|_j \leq \kappa < 1 \leq |a_n|_j \leq |a|_j. \]  Otherwise $|a_i|_j \geq 1$, so by (\ref{NASTY2})
\[ \left(1-\eta^{\frac{1}{m'}}\right) |a_i|_j \geq 1-\eta^{\frac{1}{m'}} > \kappa \]
and thus 
\[ |b_i|_j \leq \kappa + \eta^{\frac{1}{m'}} |a_i|_j < |a_i|_j \leq |a|_j. \] 
Thus we have $|b|_j \leq |a|_j$ for all $j$ with strict inequality at each Archimedean place hence $|b| < |a|$.  By minimality of $a$, $f(b) \neq 0$.  Put $c = \frac{b_n}{a_n}$; then 
\begin{equation}
\label{NASTY5}
|b-ca| \leq \prod_{j=1}^{m'} \left( C_{\delta} \eta |a_n| \right)^{\frac{-1}{m'(n-1)}} = \left( C_{\delta} \eta |a_n| \right)^{\frac{-1}{n-1}}. 
\end{equation}
Combining (\ref{ISOEQ3}) and (\ref{NASTY5}) as above yields 
\[ \max_i |a_i| = |a_n| \leq \frac{1}{\eta C_{\delta}} |3f|^{\frac{n-1}{2}}. \]
Step 4: If $R = \Z$, then it is no loss of generality to suppose that $|a| \geq 2$, as the claimed bound certainly holds otherwise.  Using Corollary \ref{3.4} there is $b = (b_1,\ldots,b_n) \in R^n$ 
with $b_n \neq 0$ such that 
\[ \forall 1 \leq i \leq n-1, |b_n \theta_i - b_i| \leq |a|^{\frac{-1}{n}}, \]
\[ |b_n| < |a|. \]
Then -- exploiting that the norm on $\Z$ is $\N$-valued -- for all $1 \leq i \leq n-1$, 
\[ |b_i| \leq |b_n \theta_i - b_i| + |b_n \frac{a_i}{a_n}| < 1 + |b_n| \leq |a|. \]
The rest of the argument is the same as in Step 3 above and leads to 
\[ |a| \leq  (3|f|)^{\frac{n-1}{2}}. \]

\end{proof}

\subsection{Some Cases of the Nullstellensatz}
\textbf{} \\ \\ \noindent
In every case, Theorem \ref{6.1} takes the form: for a suitable normed domain $(R,|\cdot|)$ and $n \in \Z^+$, there is a constant $Q(R,n)$ such that every 
nonzero isotropic $n$-ary quadratic form $f$ over a normed domain $(R,|\cdot|)$ admits an isotropic vector $|a|$ with 
\begin{equation}
\label{EXPBOUND}
|a| \leq Q(R,n) |f|^{\frac{n-1}{2}}. 
\end{equation}
When $R = \Z$ the existence of a 
bound of the form (\ref{EXPBOUND}) was shown by Cassels \cite{Cassels55}.  In his textbook \cite{Cassels}, Cassels gave an improved argument leading to the better 
bound $Q(R,n) = 3^{\frac{n-1}{2}}$.  We have essentially reproduced this argument in our Theorem \ref{6.1}c).  Cassels gives examples to show that the exponent $\frac{n-1}{2}$ cannot be 
improved upon, and thus Theorem \ref{6.1}b) is sharp up to the constant 
$Q(\Z,n)$.  Whether one can improve upon $Q(\Z,n) = 3^{\frac{n-1}{2}}$ seems to be an open question.  There is certainly no room for improvement coming from linear forms: we have used that the linear constants $C(\Z,n)$ are all equal to $1$ -- the largest possible value -- and even a little more via Theorem \ref{3.4}. 
\\ \\
By Theorem \ref{NBR2}, the hypotheses of part b) hold when $R = \Z_K$ 
for an imaginary quadratic field $K$.  A result of this form was first proved by \cite{Raghavan75}, who showed that one can take $Q(R,n) = \disc(K)^{\frac{n}{4}} 5^{\frac{n-1}{2}}$.  To apply Theorem \ref{6.1} in this case we take the square root of the canonical norm on $\Z_K$, giving $C(R,n) \geq \left(\frac{2}{\pi}\right)^{\frac{n}{2}} (\disc K)^{\frac{n}{4}}$.  Our approach gives a better constant, at least asymptotically: assuming that $|f|$ is large enough so that the ``eta factor'' in (\ref{GRANDEQ}) can be ignored, we get a 
constant arbitrarily close to $\left(\frac{2}{\pi}\right)^{\frac{n}{2}} 3^{\frac{n-1}{2}} (\disc K)^{\frac{n}{4}}$.  
\\ \\
(We admit that the eta factor in (\ref{GRANDEQ}) seems to be an artifice of the proof.  Unfortunately we do not know how to remove it, but probably someone else will.)  
\\ \\
When $K$ is a number field with more than one infinite place, the canonical norm is not metric.  This did not stop Raghavan from proving a generalization of 
Cassels's Theorem in this context: the constant he gets is $\disc(K)^{\frac{n}{2[K:\Q]}} 5^{\frac{n-1}{2}}$.  However he does not use (an equivalent norm to) the canonical norm: in fact his measure of the size of 
the coefficients is not a norm at all in our sense, as it is only submultiplicative (but satisfies the triangle inequality).  
\\ \\
Combining the Nullstellensatz with Theorem \ref{NBR2} for $R = \Z_K$ we recover a variant of Ragahvan's result.  But moreover we may take $R = \Z_{K,S}$ to be 
any $S$-integer ring.  This is a new result, but as we will see it is a natural one, being an analogue of a result of Pfister in the function field case.  
\\ \\
Turning now to the $q$-normed case of Theorem \ref{6.1}, we get cleaner results. 

\begin{cor}(Prestel \cite{Prestel87})
Let $k$ be a field\footnote{In fact Prestel assumed that $\car k \neq 2$, but as Pfister later noted, to get around this one need only not divide by $2$ in the 
definition of the associated bilinear form.}, and let $f$ be a nonzero $n$-ary quadratic form with coefficients in $k[t]$.  If $f$ is isotropic, there is an isotropic vector $v$  with $\deg v \leq \frac{n-1}{2} \deg f$.  
\end{cor}
\begin{proof}
Apply Corollary \ref{4.4} and Theorem \ref{6.1}.
\end{proof}
\noindent
Our overall method of proof of Theorem \ref{6.1} owes a lot to \cite{Prestel87}: roughly, we replaced an \emph{ad hoc} argument on linear systems over $k[t]$ with our theory of (mutinormed) linear constants.  \\ \indent
Again Prestel gives an example to show that the exponent $\frac{n-1}{2}$ in 
(\ref{EXPBOUND}) is best possible, again whether the constant is best possible 
remains open, and again there is no possible improvement coming from the theory 
of linear constants, since $c(k[t],n) = n-1$ is the largest possible value.  
\\ \indent 
In the same paper, Prestel considers the ring $R = \R[x,y]$.  Writing $\deg f$ 
for the total degree of an element of $R$, notice that for fixed $q > 1$, $|f| = q^{\deg f}$ gives an \emph{elementwise} multiplicative $q$-norm function on the UFD $R$.  It is sensible to define the linear $q$-constants $c(R,n)$ in this context -- since $R$ is not a Dedekind domain, one ought to restrict to free lattices -- and if 
$c(R,n) > -\infty$, the proof of Theorem \ref{6.1}c) would apply to give a bound 
on the degree of an isotropic vector for an isotropic quadratic form in terms 
of the total degrees of the coefficients of the form.  However, for $n = 16$ 
Prestel exhibits for each $v \in \N$ a quadratic form $f_v \in R[t_1,\ldots,t_n]$ 
with total degree $2$ and such that the least degree of an isotropic vector is 
at least $v$ \cite[Thm. 2]{Prestel87}.  Thus $\R[x,y]$ is not of linear type!

\begin{cor}
\label{6.3}
Let $C_{/k}$ be a nice projective curve over $k$, let $\infty_1,\ldots,\infty_m$ be closed points of degrees $d_1,\ldots,d_m$.  Let 
$C^{\circ} = C \setminus \{\infty_1,\ldots,\infty_m\}$, and let $k = k[C^{\circ}]$ endowed with its canonical $q$-norm of $\S$ 1.8. 
Let $f \in R[t_1,\ldots,t_n]$ be a nonzero isotropic quadratic form.  Then $f$ admits an isotropic vector $v$ with 
\begin{equation}
\label{PFISTEREQ}
 \deg v \leq (\sum_j d_j +g-1)n + \frac{n-1}{2} \deg f. 
\end{equation}
\end{cor}
\begin{proof}
Apply Theorems \ref{4.7} and \ref{6.1}.
\end{proof}
\noindent
Corollary \ref{6.3} is a variant of the Nullstellensatz of A. Pfister.  
For $f \in k(C)^{\bullet}$, let $\deg_P f$ be the degree of the polar part of $\div f$; by taking maxima we extend this notion of $\deg_P$ to vectors and matrices with coefficients in $k(C)$.  Then:

\begin{thm}(Pfister \cite{Pfister97})
\label{6.4}
With hypotheses as in Corollary \ref{6.3}, $f$ admits an isotropic vector $v$ 
with 
\begin{equation}
\deg_P v \leq (\max_{i} d_i +g-1)n + \frac{n-1}{2} \deg_P f.
\end{equation}
\end{thm}
\noindent
For $x \in k[C^{\circ}]^{\bullet}$, $\deg x$ is the sum of all of the infinite degrees $\deg_j x$ whereas $\deg_P x$ is the sum over only the non-negative 
terms $\deg_j x$, so $\deg x \leq \deg_P f$.  (Further, $\deg x$ depends on the chosen set of infinite places whereas $\deg_P f$ does not.)  When $m = 1$ 
we have $\deg_P = \deg$ and indeed Corollary \ref{6.3} and Theorem \ref{6.4} 
coincide.  For $m > 1$ the constant in Theorem \ref{6.4} is smaller than that 
of Corollary \ref{6.3}, but because the norms are different the results do not 
appear to be directly comparable.  However, Pfister himself showed that a variant of Theorem \ref{6.4} follows easily from the common special case $m = 1$ by a short argument involving the Riemann-Roch Theorem.  Thus the following 
result is also a corollary of our Nullstellensatz.

\begin{cor}(Pfister \cite[p. 230]{Pfister97})
\label{6.5}
With hypotheses as in Corollary \ref{6.3}, $f$ admits an isotropic vector $v$ with 
\[ \deg_P v \leq \frac{3n-1}{2}(\min_j d_j + g - 1) + \frac{n-1}{2} \deg_P f. \]
\end{cor}
\noindent
Thus we recover Pfister's Theorem \ref{6.4} up to a different value of 
the constant $Q(R,n)$.  In fact, as Pfister himself remarks, the constant 
given in Corollary \ref{6.5} is sometimes worse and sometimes better than 
that of Theorem \ref{6.4}.

\section{Quadratic Forms: The Small Multiple Theorem}
\noindent
An ideal $I$ in a ring $R$ is \textbf{odd} if it is coprime to $2R$.  An element 
$x$ of $R$ is odd if the principal ideal $(x)$ is odd.

\begin{thm}
\label{MAGICLATTICETHM}
Let $R$ be a Dedekind domain with fraction field $K$, let $q(x) = q(x_1,\ldots,x_n)$ be a nondegenerate quadratic form over $R$, and let $I$ be an odd ideal of $R$ which is coprime to $\disc q$.  We further assume: \\
(H) The base change of $q$ to $R/I$ is hyperbolic, i.e., isomorphic to $\bigoplus_{i=1}^{\frac{n}{2}} \mathbb{H}$. Then: \\
a) There is an $R$-sublattice $\Lambda_I \subset R^n$ such that:  \\
(i) We have $R^n/\Lambda_I \cong (R/I)^{\frac{n}{2}}$ and thus $\chi(\Lambda_I) = I^{\frac{n}{2}}$.  \\
(ii) We have $q(v) \equiv 0 \pmod I$ for all $v \in \Lambda_I$.  \\
b) The $R$-module $\Lambda_I$ is free iff $I^{\frac{n}{2}}$ is principal. \\
c) Each of the following conditions implies (H): \\
(H1) $n = 2$ and $-d(q)$ is a square in $R/I$.  \\
(H2) Every residue field of $R/I$ has $u$-invariant at most $2$ (e.g. this holds when every residue field is finite), $n$ is even and $(-1)^{\frac{n}{2}} d(q)$ 
is a square in $R$.   
\end{thm}
\begin{proof} 
a) Step 1: We suppose $I = \pp^e$ is an odd prime power.  Then $k := R/\pp$ 
is a field of characteristic different from $2$.  Let $R_{\pp}$ be the 
completion of $R$ at $\pp$; then $R_{\pp}$ is a nondyadic CDVR with fraction field $K_{\pp}$, and since $m$ is prime to $\Disc q$, the base change $\hat{q}$ of $q$ to $R_{\pp}$ is \textbf{nonsingular}.  Since the reduction of $\hat{q}$ modulo $\pp$ is isotropic, by Hensel's Lemma so is $\hat{q}$.  Thus $\hat{q}_{/K_{\pp}}$ is universal and similar to a Pfister form, hence is itself an isotropic Pfister form.  Every isotropic Pfister form is hyperbolic, so 
$\hat{q}_{K_{\pp}} \cong_{K_{\pp}} \bigoplus_{i=1}^{\frac{n}{2}} \mathbb{H}$.  
Since $\hat{q}$ is nonsingular, it follows that $\hat{q} \cong_{R_{\pp}} \bigoplus_{i=1}^{\frac{n}{2}} \mathbb{H}$ (e.g. \cite[Thm. 1.6.13]{Scharlau}), and thus $q_{/R/(m)} \cong \bigoplus_{i=1}^{\frac{n}{2}} \mathbb{H}$.  If the $i$th copy of the hyperbolic plane is the free $R/I$-module with basis 
$e_i,f_i$, put $M = \langle e_1,\ldots,e_{\frac{n}{2}} \rangle_{R/I}$.  Let $\varphi: 
R^n \ra (R/I)^n$ be the canonical map, and let $\Lambda_I = \varphi^{-1}(M)$.  
Then $\Lambda_I$ is an $R$-submodule of $\Lambda_I$ with finite length quotient, so it is an $R$-lattice in $K^n$.  Clearly $\chi(R^n/\Lambda_I) = 
I^{\frac{n}{2}}$, and by construction, 
$q(v) \equiv 0 \pmod{m}$ for all $v \in \Lambda_I$, so this completes the proof 
of Theorem \ref{MAGICLATTICETHM} in this case.  \\
Step 2: Suppose $I = \pp_1^{e_1} \cdots \pp_r^{e_r}$.  For $1 \leq i \leq r$, put $I_i = \pp_i^{e_i}$.  By Step 1, for $1 \leq i \leq r$ there is a sublattice $\Lambda_i \subset R^n$ such that $\chi(R^n/\Lambda_i) = I_i^{\frac{n}{2}}$ and $q|_{\Lambda_i} \equiv 0 \pmod{I_i}$.  Put $\Lambda_I = \bigcap_{i=1}^r \Lambda_i$.  Then 
$\Lambda_I$ is a sublattice of $R^n$; by the Chinese Remainder Theorem $\chi(R^n/\Lambda_I) = \prod_{i=1}^r \chi(R^n/\Lambda_i) = 
I^{\frac{n}{2}}$ and $q(v) \equiv 0 \pmod{I}$ for all $v \in \Lambda_I$. \\
b) This follows easily from the fact that $R^n/\Lambda_I \cong (R/I)^{\frac{n}{2}}$.  
\end{proof}

\begin{thm}
\label{MAINTHM}
\label{7.2}
Let $(R,|\cdot|)$ be a multimetric linear type normed Dedekind domain, with fraction field $K$.  Let $f = f(t_1,\ldots,t_n) \in R[t_1,\ldots,t_n]$ be an anisotropic quadratic form.  Let $d$ be an odd 
element of $R$ which is coprime to $\disc q$.  We suppose \textbf{hypothesis (H)}: the base change of $q$ to $R/(d)$ is isomorphic to $\bigoplus_{i=1}^{\frac{n}{2}} \mathbb{H}$. \\
a) If $|\cdot|_j$ is Archimedean for at least one $j$, then for any $0 < c < C_M(R,n)$ there is $v \in R^n$ and $k \in R$ such that 
\begin{equation}
\label{MAINEQ2}
q(v) = kd, \ 0 < |k| \leq c^{\frac{-2}{n}} |f|.
\end{equation}
b) If the norm is of $q$-type, there is $v \in R^n$ 
and $k \in R^{\bullet}$ such that 
\[ q(v) = kd, \ \deg k \leq \deg f + \frac{2(mn-1)}{n} - \frac{2c_M(R,n)}{n}. \]
\end{thm}
\begin{proof} By Theorem \ref{MAGICLATTICETHM} there is an integral lattice 
$\Lambda_d$ with $\Covol \Lambda_d = |d|^{\frac{n}{2}}$ and such that $q(v) \equiv 0 \pmod{d}$ for all $v \in \Lambda_d$.  \\
a) Suppose that $|\cdot|_j$ is Archimedean for $1 \leq j \leq m'$ and ultrametric for $m' < j \leq m$.  For all $1 \leq i \leq n, \ 1 \leq j \leq m'$, take \[\epsilon_{i,j} = 
\epsilon = c^{\frac{-1}{m'n}} |d|^{\frac{1}{2m'}}. \]  
For all $1 \leq i \leq n, \ m' < j \leq m$, take 
\[ \epsilon_{i,j} = 1. \]
Then $\Covol \Lambda_d \leq c \prod_{i,j} \epsilon_{i,j}$, so there is $v = (v_1,\ldots,v_n) \in \Lambda_d^{\bullet}$ with $|v_i|_j \leq c^{\frac{-1}{m'n}} |d|^{\frac{1}{2m'}}$ 
for all $1 \leq i \leq n$, $1 \leq j \leq m'$ and $|v_i|_j \leq 1$ for all 
$1 \leq i \leq n$, $m' < j \leq m$.  Then 
\[ |f(v)| = |\sum_{i,k} a_{ik} v_i v_k| = \prod_{j=1}^m |\sum_{i,k} a_{ik} v_i v_k|_j \leq \prod_{j=1}^{m'} \sum_{i,k} |a_{ik} v_i v_k|_j \prod_{j=m'+1}^m \max_{ik} |a_{ik} v_i v_k|_j\] 
\[ \leq \epsilon^{2m'} \prod_{j=1}^{m'} \sum_{i,k} |a_{ik}|_j \prod_{j=m'+1}^m 
\max_{ik} |a_{ik}|_j= 
\epsilon^{2m'} |f| = c^{\frac{-2}{n}} |d| |f|. \]
Writing $f(v) = kd$ and using $|f(v)| = |k| |d|$, we get 
\[ |k| \leq c^{\frac{-2}{n}} |f|. \] 
b) For all $1 \leq i \leq n, \ 1 \leq j \leq m$ we take 
\[e_{ij} = e =  \lceil \frac{ \frac{n}{2} \deg d - c_M(R,n)}{mn} \rceil. \]
Then $\covol \Lambda_d = \frac{n}{2} \deg d \leq c_M(R,n) + \sum_{i,j} e_{ij}$, 
so there is $v = (v_1,\ldots,v_n) \in \Lambda_d^{\bullet}$ such that 
\[ \forall 1 \leq i \leq n, \ 1 \leq j \leq m, \ \deg_j v_i \leq \lceil \frac{ \frac{n}{2} \deg d - c_M(R,n)}{mn} \rceil  \] \[ \leq \frac{ \frac{n}{2} \deg d - c_M(R,n)}{mn} + \frac{mn-1}{mn}. \]
Then 
\[ \deg f(v) = \sum_{j=1}^m \deg_j f(v) = \sum_{j=1}^m \deg_j \sum_{ik} a_{ik} v_i v_k \]
\[ \leq \sum_{j=1}^m \max_{ik} (\deg_j a_{ik} + \deg_j v_i + \deg_j v_k) 
\leq \sum_{j=1}^m \max_{ik} (\deg_j a_{ik} + 2e) \] \[ =  \deg f + 2me 
\leq  \deg f + \frac{2(mn-1)}{n} - \frac{2c_M(R,n)}{n} + \deg d. \]
Writing $f(v) = kd$ for $k \neq 0$ and using $\deg f(v) = \deg k + \deg d$, we get 
\[ \deg k \leq  \deg f + \frac{2(mn-1)}{n} - \frac{2c_M(R,n)}{n}. \]
\end{proof}

\begin{thm}
\label{7.3}
Let $f = \sum_{i,j} a_{ij} t_i t_j \in \Z[t_1,\ldots,t_n]$ be an anisotropic quadratic form.  Let $d \in \Z^{\bullet}$ be an odd integer coprime to $\disc f$ and such that the base change of $f$ to $\Z/d\Z$ is hyperbolic.  Then there is $v \in \Z^n$ such that 
\begin{equation}
\label{MAINEQ3}
 f(v) = kd, \ 0 < |k| < |f|. 
\end{equation}
\end{thm}
\begin{proof} 
Using $C(\Z,n) = 1$ for all $n$ and a simple limiting argument we get 
(\ref{MAINEQ3}) with ``$\leq$''.  Using instead the sharper Theorem \ref{ZLINTYPE2} one extracts a strict inequality.  
\end{proof}
\noindent
We refer the interested reader to \cite{Mordell66} for examples of the use of Theorem \ref{7.3} to prove representation theorems for binary integral quadratic forms.  Here we will content ourselves with one striking example. 
\\ \\
Example (Brauer-Reynolds \cite{Brauer-Reynolds}): Let $f = x^2 + y^2 + z^2 + w^2$ over $\Z$.  For every 
odd positive integer $d$, there is $v \in \Z^4$ with $f(v) = kd$, $0 < k < 4$.  
We can deduce Lagrange's Theorem that $q$ represents all 
positive integers.  Indeed, since $f(\Z^4)$ is closed under multiplication 
and certainly contains $1$ and $2$, it suffices to show that $f$ represents every odd prime $p$.  We know that there is $v$ with 

\begin{equation}
\label{LAGRANGE1}
 x^2 + y^2 + z^2 + w^2 = p
\end{equation}
or 
\begin{equation}
\label{LAGRANGE2}
x^2+ y^2 + z^2 + w^2 = 2p 
\end{equation}
or 
\begin{equation}
\label{LAGRANGE3}
x^2 + y^2 + z^2 + w^2 = 3p. 
\end{equation}
If (\ref{LAGRANGE1}) holds, we're done.  If (\ref{LAGRANGE2}) holds, 
parity considerations show: after reordering the variables we may 
assume $x \equiv y \pmod{2}$ and $z \equiv w \pmod{2}$ and then 
\[ p = \left( \frac{x+y}{2} \right)^2 + \left( \frac{x-y}{2} \right)^2 + 
\left( \frac{z+w}{2} \right)^2 + \left( \frac{z-w}{2} \right)^2. \]
If (\ref{LAGRANGE3}) holds, then one of $x,y,z,w$ -- say $x$ -- 
must be divisible by $3$; replacing $y$, $z$, $w$ with their negatives 
if necessary we may assume $y \equiv z \equiv w \pmod{3}$, and then 
\[ p = \left( \frac{y+z+w}{3} \right)^2 + \left( \frac{x+z-w}{3} \right)^2 
+ \left( \frac{x-y+w}{3} \right)^2 + \left( \frac{x+y-z}{3} \right)^2. \]


\begin{thm}
\label{7.5}
Let $k$ be a field and $C/k$ be a nice projective curve of genus $g$, and let $C^{\circ} = C \setminus \{\infty_1,\ldots,\infty_m\}$ be the affine curve obtained by removing the given $m$ closed points, of degrees $d_1,\ldots,d_m$.   Let $R = k[C^{\circ}]$ be its affine coordinate ring.  Let $f \in R[t_1,\ldots,t_n]$ be an anisotropic quadratic form.  Let $d \in R^{\bullet}$ 
be odd and coprime to $\disc f$ and such that the base change of $f$ to 
$R/dR$ is hyperbolic.  Then there is $v \in R^n$ such that 
\begin{equation}
f(v) = kd, \ 0 \leq \deg k \leq \deg f + 2(\sum_{j=1}^m d_j + g - 1).
\end{equation}
\end{thm}
\begin{proof} 
Apply Theorems \ref{4.7} and \ref{7.2}.
\end{proof}


\begin{thebibliography}{CLRR80}

\bibitem[A]{Artin} E. Artin, \emph{Algebraic numbers and algebraic functions}. Gordon and Breach Science Publishers, New York-London-Paris 1967.


\bibitem[BR51]{Brauer-Reynolds} A. Brauer and R.L. Reynolds, \emph{On a theorem of Aubry-Thue}.  Canadian J. Math. 3 (1951), 367–-374. 



\bibitem[BW66]{Butts-Wade} H.S. Butts and L.I. Wade, \emph{Two criteria for Dedekind domains}. Amer. Math. Monthly 73 (1966), 14--21.

\bibitem[C]{Cassels} J.W.S. Cassels, \emph{An introduction to the geometry of numbers}. Corrected reprint of the 1971 edition. Classics in Mathematics. Springer-Verlag, Berlin, 1997.

\bibitem[Ca55]{Cassels55} J.W.S. Cassels, \emph{Bounds for the least solutions of homogeneous quadratic equations}. Proc. Cambridge Philos. Soc. 51 (1955), 262-–264. 


\bibitem[Ch12]{Chonoles} Z. Chonoles, \emph{Hermite's Theorem for Function Fields}. Honors Thesis, Brown University, 2012.  

\bibitem[CF]{Cassels-Frohlich} Algebraic number theory.
Proceedings of an instructional conference organized by the London Mathematical Society (a NATO Advanced Study Institute) with the support of the Inter national Mathematical Union. Edited by J. W. S. Cassels and A. Fr\"ohlich Academic Press, London; Thompson Book Co., Inc., Washington, D.C. 1967.

\bibitem[CL]{CL} J.-P. Serre, \emph{Corps Locaux}, Hermann, Paris, 1962.

\bibitem[CL70]{Chew-Lawn70} K.L. Chew and S. Lawn, 
\emph{Residually finite rings}.
Canad. J. Math. 22 1970 92--101.

\bibitem[Cl12]{ADCI} P.L. Clark, \emph{Euclidean Quadratic Forms and ADC-forms I}. Acta Arithmetica 154 (2012), 137--159.

\bibitem[Co]{Conrad} B. Conrad, \emph{The Lattice of $S$-Integers}. \\ {\tt  http://math.stanford.edu/$\sim$conrad/249BPage/handouts/Sintlattice.pdf}










\bibitem[GoN1]{GoN1} P.L. Clark, J. Hicks, H. Parshall and K. Thompson, \emph{GoNI: Primes represented by binary quadratic forms}. . INTEGERS 13 (2013), A37, 18 pp.

\bibitem[GoN2]{GoN2} P.L. Clark, J. Hicks, K. Thompson and N. Walters, \emph{GoNII: Universal quaternary quadratic forms}. INTEGERS 12 (2012), A50, 16 pp.

\bibitem[GoN3]{GoN3} J. Hicks and K. Thompson, \emph{GoNIII: More universal 
quaternary forms}, preprint.  








\bibitem[Ic97]{Icaza97} M.I. Icaza, \emph{Hermite constant and extreme forms for algebraic number fields}. J. London Math. Soc. (2) 55 (1997), 11-–22. 

\bibitem[Ku]{Kuhlmann} F.-V. Kuhlmann, \emph{The defect}. Commutative algebra--Noetherian and non-Noetherian perspectives, 277–-318, Springer, New York, 2011.

\bibitem[L]{Lang} S. Lang, \emph{Algebraic Number Theory}. Second edition. Graduate Texts in Mathematics, 110. Springer-Verlag, New York, 1994.

\bibitem[LM]{Larsen-McCarthy} M.D. Larsen and P.J. McCarthy,
\emph{Multiplicative theory of ideals.}
Pure and Applied Mathematics, Vol. 43. Academic Press, New York-London, 1971.

\bibitem[LM72]{Levitz-Mott72} K.B. Levitz and J.L. Mott,
\emph{Rings with finite norm property}. Canad. J. Math.  24  (1972), 557--565.

\bibitem[M]{Matsumura} H. Matsumura, \emph{Commutative ring theory}.
Translated from the Japanese by M. Reid. Second edition. Cambridge Studies in Advanced Mathematics, 8. Cambridge University Press, Cambridge, 1989.

\bibitem[Ma41]{Mahler41} K. Mahler, \emph{An analogue to
Minkowski's geometry of numbers in a field of series}. Ann. of Math. (2) 42, (1941), 488-–522.







\bibitem[Mo51]{Mordell51} L.J. Mordell, \emph{On the equation $ax^2+by^2-cz^2=0$}. Monatsh. Math. 55 (1951), 323-–327. 

\bibitem[Mo66]{Mordell66} L.J. Mordell, \emph{Solvability of the equation $ax^2+by^2=p$}. J. London Math. Soc. 41 (1966), 517–-522.







\bibitem[Pf97]{Pfister97} A. Pfister, \emph{Small zeros of quadratic forms over algebraic function fields}. Acta Arith. 79 (1997), 221-–238. 

\bibitem[Pr87]{Prestel87} A. Prestel, \emph{On the size of zeros of quadratic forms over rational function fields}. J. Reine Angew. Math. 378 (1987), 101-–112. 

\bibitem[Ra75]{Raghavan75} S. Raghavan, \emph{Bounds for minimal solutions of Diophantine equations}. Nachr. Akad. Wiss. G\"ottingen Math.-Phys. Kl. II 1975, no. 9, 109–-114. 



\bibitem[S]{Samuel} P. Samuel, \emph{Th\'eorie alg\'ebrique des nombres}. Hermann, Paris 1967. 

\bibitem[Sc]{Scharlau} W. Scharlau, \emph{Quadratic and Hermitian forms}. Grundlehren der Mathematischen Wissenschaften 270. Springer-Verlag, Berlin, 1985.


\bibitem[SK68]{Stevens-Kuty} H. Stevens and L. Kuty, \emph{Applications of an elementary theorem to number theory}. Arch. Math. (Basel) 19 (1968), 37–-42. 

\bibitem[Th02]{Thue}
A. Thue, \emph{Et par antydninger ti1 en taltheoretisk metode}, Kra. Vidensk. Selsk.
Forh. 7 (1902), 57-75.



\bibitem[To41]{Tornheim} L. Tornheim, \emph{Linear forms in function fields}.
Bull. Amer. Math. Soc. 47 (1941), 126–-127. 

\bibitem[Vi27]{Vinogradov} I.M. Vinogradov, \emph{On a general theorem concerning the distribution of the residues and
non-residues of powers}. Trans. Amer. Math. Soc. 29 (1927), 209--17.


\end{thebibliography}
\end{document}